\documentclass[a4paper,11pt]{amsart}
\usepackage{mathrsfs}
\usepackage{appendix}
\usepackage{amssymb}
\usepackage{fancyhdr}
\usepackage{charter}
\usepackage{typearea}
\usepackage{color}
\usepackage{amsmath,amstext,amsthm,amscd}
\usepackage{url}

\usepackage[a4paper,top=3cm,bottom=3cm,left=3cm,right=3cm]{geometry}

\makeatletter
      \def\@setcopyright{}
      \def\serieslogo@{}
      \makeatother

\newcommand{\Complex}{\mathbb C}
\newcommand{\Real}{\mathbb R}
\newcommand{\N}{\mathbb N}
\newcommand{\ddbar}{\overline\partial}
\newcommand{\pr}{\partial}
\newcommand{\ol}{\overline}
\newcommand{\Td}{\widetilde}
\newcommand{\norm}[1]{\left\Vert#1\right\Vert}
\newcommand{\abs}[1]{\left\vert#1\right\vert}

\newcommand{\set}[1]{\left\{#1\right\}}
\newcommand{\To}{\rightarrow}

\newcommand{\R}{\mathbb{R}}
\newcommand{\C}{\mathbb{C}}

\theoremstyle{plain}
\newtheorem{theorem}{Theorem}[section]
\newtheorem{lemma}[theorem]{Lemma}
\newtheorem{corollary}[theorem]{Corollary}
\newtheorem{proposition}[theorem]{Proposition}

\newtheorem{definition}[theorem]{Definition}

\newtheorem{example}[theorem]{Example}

\numberwithin{equation}{section}

\begin{document}
\title[{$G$-equivariant embedding theorems for CR manifolds of high codimension}]
{$G$-equivariant embedding theorems for CR manifolds of high codimension}
\author[Kevin Fritsch]{Kevin Fritsch}
\address{Ruhr-Universit\"at Bochum
Fakult\"at f\"ur Mathematik, Germany}
\email{kevin.fritsch@rub.de}

\author[Hendrik Herrmann]{Hendrik Herrmann}
\address{Mathematical Institute, University of Cologne, Weyertal 86-90, 50931 Cologne, Germany}
\email{heherrma@math.uni-koeln.de or post@hendrik-herrmann.de}
\author[Chin-Yu Hsiao]{Chin-Yu Hsiao}
\address{Institute of Mathematics, Academia Sinica and National Center for Theoretical Sciences, Astronomy-Mathematics Building, No. 1, Sec. 4, Roosevelt Road, Taipei 10617, Taiwan}
\thanks{Chin-Yu Hsiao was partially supported by Taiwan Ministry of Science and Technology project 107-2115-M-001-012-MY2 and Academia Sinica Career Development
Award. }
\thanks{Kevin Fritsch was partially supported by the CRC TRR 191: ``Symplectic Structures in Geometry, Algebra and Dynamics''.}
\thanks{Kevin Fritsch and Hendrik Herrmann would like to thank the Mathematical Institute, Academia Sinica for  the hospitality,  wonderful accomodations and financial support during their visits in February/March and September/October 2018.}
\email{chsiao@math.sinica.edu.tw or chinyu.hsiao@gmail.com}

\begin{abstract}
Let $(X,T^{1,0}X)$ be a $(2n+1+d)$-dimensional compact CR manifold with codimension $d+1$, $d\geq1$, and let $G$ be 
a $d$-dimensional compact Lie group with CR action on $X$ and $T$ be a globally
defined vector field on $X$ such that $\Complex TX=T^{1,0}X\oplus T^{0,1}X\oplus\Complex T\oplus\Complex\underline{\mathfrak{g}}$, where $\underline{\mathfrak{g}}$ is the space of vector fields on $X$ induced by the Lie algebra of $G$. In this work, we show that if $X$ is strongly pseudoconvex in the direction of $T$ and $n\geq 2$, then there exists a $G$-equivariant CR embedding of $X$ into $\Complex^N$, for some $N\in\mathbb N$. We also establish a CR orbifold version of Boutet de Monvel's embedding theorem. 
\end{abstract}

\maketitle \tableofcontents

\section{Introduction and statement of the main results}\label{s-gue180710}

The embedding of CR manifolds in general is a subject with
long tradition. One paradigm is the embedding theorem of compact strongly pseudoconvex
CR manifolds of codimension one. A famous theorem of Louis Boutet de Monvel~\cite{BdM1:74b}
asserts that such manifolds can be embedded
by CR maps into the complex Euclidean space, provided the dimension of the manifold is
greater than or equal to five. 

The goal of this paper is to study 
equivariant embeddings of CR manifolds with high codimension
which admit a compact Lie group action $G$. Let us see a simple example and explain briefly our motivation. 
Let $G$ be a compact Lie group and assume that $G$ acts holomorphically on a K\"ahler manifold $(M,\omega)$, where $\omega$ denotes the K\"ahler form on $M$. Let $\mu: M \to \mathfrak{g}^*$ be a moment map induced by $\omega$.  Assume that $0 \in \mathfrak{g}^*$ is regular. Then, \emph{$\mu^{-1}(0)$ is a CR manifold and the Lie group $G$ acts on $\mu^{-1}(0)$}. The study of $G$-equivariant CR embeddability for $\mu^{-1}(0)$ is closely related to some important problems in Mathematical physics and geometric quantization theory. It should be noticed that $\mu^{-1}(0)$ is a CR manifold \emph{with high codimension in general and the action of $G$ is transversal and CR on $X$} (see Example \ref{HighCodimensionExample1}). Therefore, it is very natural to study $G$-equivariant CR embedding problems for CR manifolds with arbitrary codimension. When the codimension of $X$ is one, the problems about $G$-equivariant CR embedding and $G$-equivariant Szeg\H{o} kernels were studied in~\cite{HHL16},~\cite{HLM16} ~\cite{HHL17} and~\cite{Hsiao18}.
In this paper, we consider a more general situation; we do not assume that $G$ is transversal. 
 We consider a $(2n+1+d)$-dimensional compact CR manifold $(X,T^{1,0}X)$ with codimension $d+1$, $d\geq1$, and let $G$ be a $d$-dimensional compact Lie group with a locally free CR action on X. Let $T$ be a globally
defined vector field on $X$ ($T$ is not necessarily CR) such that $\Complex TX=T^{1,0}X\oplus T^{0,1}X\oplus\Complex T\oplus\Complex\underline{\mathfrak{g}}$ holds, where $\underline{\mathfrak{g}}$ is the space of vector fields on $X$ induced by the Lie algebra of $G$. In this work, we show that if $X$ is strongly pseudoconvex in the direction of $T$ and $n\geq 2$, then there exists a $G$-equivariant CR embedding of $X$ into $\Complex^N$ for some $N\in\mathbb N$. Since the action $G$ is locally free  $X/G$ is a strongly pseudoconvex CR orbifold (see Definition~\ref{d-gue181015} for the definition of CR orbifolds). Actually, every compact strongly pseudoconvex CR orbifold can be obtained in this way (see Theorem~\ref{t-gue181015}). As an application of our result we establish a CR orbifold version of Boutet de Monvel's embedding theorem which is interesting in itself because it plays an important role in orbifold geometry. 

We now formulate our main result. We refer the reader to Section~\ref{s:prelim} for some standard notations and terminology used here.
Let $(X, T^{1,0}X)$ be a $(2n+1+d)$-dimensional compact and orientable CR manifold of codimension $d+1$, $d\geq1$, where $T^{1,0}X$ is a CR structure on $X$, that is, $T^{1,0}X$ is a subbundle of rank $n$ of the complexified tangent bundle $\mathbb{C}TX$, satisfying $T^{1,0}X\cap T^{0,1}X=\{0\}$, where $T^{0,1}X=\overline{T^{1,0}X}$, and $[\mathcal V,\mathcal V]\subset\mathcal V$, where $\mathcal V=C^\infty(X, T^{1,0}X)$. In this work, we assume that $X$ admits a group action of a $d$-dimensional compact Lie group $G$.  We assume throughout that  this action is CR (see Definition~\ref{d-gue180710}). Let $T\in C^\infty(X,TX)$ be a global 
defined vector field such that 
 \[\Complex TX=T^{1,0}X\oplus T^{0,1}X\oplus\Complex T\oplus\Complex\underline{\mathfrak{g}},\]
  where $\underline{\mathfrak{g}}$ is the space of vector fields on $X$ induced by the Lie algebra of $G$. Let $\omega_0(x)\in C^\infty(X,T^*X)$ be the globally real one form on $X$ such that 
\[
\begin{split}
&\langle\,\omega_0\,,\,V\,\rangle=0,\ \ \forall V\in T^{1,0}X\oplus T^{0,1}X\oplus\Complex\underline{\mathfrak{g}}, \\
&\langle\,\omega_0\,,\,T\,\rangle=-1\ \ \mbox{on $X$}.
\end{split}\]
For $x\in X$, let $\mathcal{L}_{\omega_0,x}$ be the Levi form with respect to $\omega_0$ at $x$ (see Definition~\ref{d-gue180710I}). We say that $X$ is strongly pseudoconvex in the direction of $T$ if the Levi form $\mathcal{L}_{\omega_0,x}$ is positive definite at every point $x$ of $X$. In Section~\ref{Sec:Examples}, we give several examples to motivate the study of high codimension CR manifolds.

We now introduce the concept of $G$-finite smooth (CR) functions. Let 
\[\mathcal{R}=\set{\mathcal{R}_m;\, m=1,2,\ldots}\]
denote the set of all irreducible unitary representations of the group $G$, including only one representation from each equivalence class (see the discussion in the beginning of Section~\ref{s-gue180710mI}). 
For each $\mathcal{R}_m$, we  write $\mathcal{R}_m$ as a matrix $\left(\mathcal{R}_{m,j,k}\right)^{d_m}_{j,k=1}$, where $d_m$ is the dimension of $\mathcal{R}_m$. Fix a Haar measure $d\mu(g)$ on $G$ so that $\int_Gd\mu(g)=1$. Take an irreducible unitary representation $\mathcal{R}_m$, for every $g\in G$, put 
\[\chi_m(g):={\rm Tr\,}\left(\mathcal{R}_{m,j,k}(g)\right)^{d_m}_{j,k=1}=\sum^{d_m}_{j=1}\mathcal{R}_{m,j,j}(g).\]
Let $u\in C^\infty(X)$ be a smooth function. The $m$-th Fourier component of $u$ is given by 
\[u_m(x):=d_m\int_G (g^*u)(x)\ol{\chi_m(g)}d\mu(g)\in C^\infty(X)\]
(see Definition~\ref{d-gue180712cm}). For every $m\in\mathbb N$, put 
\[C^\infty_m(X):=\set{f\in C^\infty(X);\, \mbox{there is a $F\in C^\infty(X)$ such that $f=F_m$ on $X$}}.\]
We will show in Corollary~\ref{c-gue180716ap} that $u$ lies in $ C^\infty_m(X)$ if and only if $u=u_m$ holds on $X$. Given a smooth function $f\in C^\infty(X)$, we say that $f$ is a $G$-finite smooth function on $X$ if $f=\sum^K_{j=1}f_j$ with $K\in\mathbb N$ and $f_j\in C^\infty_{m_j}(X)$, for some $m_j\in\mathbb N$, $j=1,\ldots,K$. Let $C^\infty_G(X)$ be the set of all $G$-finite smooth functions on $X$. Note that by construction we have that \(\text{span}_{\C}Gf\) is finite dimensional for any \(f\in C^\infty_G(X)\). 

Let $\ddbar_b: C^\infty(X)\To\Omega^{0,1}(X)$ be the tangential Cauchy-Riemann operator (see \eqref{e-gue180710m}). Put 
\begin{equation}\label{e-gue180803mpy}
H^0_b(X):=\set{u\in C^\infty(X);\, \ddbar_bu=0}
\end{equation}
and set 
\begin{equation}\label{e-gue180803mpyI}
H^0_{b,G}(X):=H^0_b(X)\bigcap C^\infty_G(X).
\end{equation}
For a smooth function $u\in C^\infty(X)$, we say that $u$ is a $G$-finite smooth CR function if $u\in H^0_{b,G}(X)$. The main result of this work is the following.

\begin{theorem}\label{t-gue180803m}
Let $(X, T^{1,0}X)$ be a $(2n+1+d)$-dimensional compact and orientable CR manifold of codimension $d+1$, $d\geq1$. Assume that $X$ admits a CR action of a $d$-dimensional compact Lie group $G$.  Let $T$ be a globally
defined vector field on $X$ such that $\Complex TX=T^{1,0}X\oplus T^{0,1}X\oplus\Complex T\oplus\Complex\underline{\mathfrak{g}}$, where $\underline{\mathfrak{g}}$ is the space of vector fields on $X$ induced by the Lie algebra of $G$. If $X$ is strongly pseudoconvex in the direction of $T$ and $n\geq 2$, then we can find $G$-finite smooth CR functions $f_j\in H^0_{b,G}(X)$, $j=1,2,\ldots,N$, $N\in\mathbb N$, such that the map
\[\begin{split}
F\colon X&\To\Complex^N\\
x&\mapsto (f_1(x),\ldots,f_N(x))
\end{split}\]
is an embedding. 
\end{theorem}
To obtain a \(G\)-equivariant CR embedding we need to show that \(F\) in Theorem~\ref{t-gue180803m} can be chosen equivariant (see Lemma~\ref{t-gue180807}) in a way that \(F(X)\) becomes a CR submanifold of \(\Complex^N\) (see Theorem~\ref{thm:CREmbedding}). For the general high codimension case, this is not obvious. 

We have a $G$-action on $H^0_b(X)$ given by $gf(x) := f(g^{-1}x)$. 
One may verify that if $f \in H^0_{b,G}(X)$, then the $G$-orbit through $f$ is contained in a finite-dimensional subspace.
We have a $G$-action on the dual $H^0_b(X)^*$ also given by $g\lambda(f) = \lambda(g^{-1}f)$.
Now if $F \colon X \rightarrow \Complex^N$, $x \mapsto (f_1 (x), \ldots, f_N(x))$ is an embedding with $f_i \in H^0_{b,G}(X)$, we define 
$W:={\rm span\,}\set{\bigcup^N_{j=1}Gf_j}$
as the smallest finite-dimensional space in $H^0_b(X)$ containing all $G$-orbits through the $f_i$.
We claim that the CR map 
\begin{eqnarray*}
\hat{F} \colon X &\To& W^* \\
x &\mapsto& ( h \mapsto h(x))
\end{eqnarray*}
is a $G$-equivariant embedding. The equation $gx \mapsto (h \mapsto h(gx)) = (h \mapsto (g^{-1}h) (x)) = g(h \mapsto h(x))$ shows the equivariance. Now let $h_1,...,h_M$ be a basis for $W$ and $h_1^*, \ldots, h_M^*$ be the dual basis.
We have $\hat{F}(x) = \sum_i \alpha_i(x) h_i^*$ with $h_j(x) = \varphi(x) (h_j) = \sum_i \alpha_i(x) h_i^* (h_j) = \alpha_j (x)$.
From this discussion, we get the following $G$-equivariant embedding result.
\begin{lemma}\label{t-gue180807}
With the same assumptions and notations used in Theorem~\ref{t-gue180803m}, the $G$-equivariant CR map $\hat F: X\To\Complex^{M}$ is an embedding. 
\end{lemma}
In general, given an arbitrary real submanifold \(X'\) of \(\Complex^{M}\) it is not true that \(X'\) is a CR  submanifold in the sense that  \(\C TX'\cap T^{1,0}\C^M\) defines a CR structure on \(X'\).
Since the embedding \(\hat{F}\) is a CR map, we find \(d\hat{F}(T^{1,0}X)\subset T^{1,0}\C^M\) and hence \(d\hat{F}(T^{1,0}X)\) defines a CR structure on \(\hat{F}(X)\) which is contained in \(\C T\hat{F}(X)\cap T^{1,0}\C^M\). The following theorem shows that the $G$-equivariant embedding \(\hat{F}\) can be chosen to be a CR embedding, that is, \(\hat{F}(X)\) is CR submanifold of \(\C^M\) and its induced CR structure  coincides with \(d\hat{F}(T^{1,0}X)\) (see Section~\ref{Sec:InducedCRStructures}).

\begin{theorem}\label{thm:CREmbedding}
With the same assumptions and notations used in Theorem~\ref{t-gue180803m}, there exists a \(G\)-equivariant CR embedding $ \hat{F}: X\To\Complex^{M}$ of \(X\) into \(\C^M\) for some \(M\in\N\). In particular, \(\hat{F}\) is a smooth embedding and \(\hat{F}(X)\) is a CR submanifold of \(\C^M\) with \[d\hat{F}\left(T^{1,0}X\right)=\C T\hat{F}(X)\cap T^{1,0}\C^N\]
where \(d\hat{F}\) denotes the differential of the map \(\hat{F}\), \(\C T\hat{F}(X)\) the complexified tangent space of the submanifold \(\hat{F}(X)\) and \(T^{1,0}\C^N\) the standard complex structure on \(\C^N\).
\end{theorem} 

Before we state our results on  CR orbifold embeddings let us say some words on the importance of the positivity assumption in Theorem~\ref{t-gue180803m}. That assumption is roughly speaking the existence of a non vanishing real one form \(\omega_0\) with \(\omega_0(T^{1,0}X\oplus T^{1,0}X)=0\) such that
\begin{itemize}
	\item[(i)] \(-\frac{1}{2i}d\omega_0\) induces a Hermitian metric on \(T^{1,0}X\),
	\item [(ii)] \(\underline{\mathfrak{g}}\) is annihilated by \(\omega_0\), that is, \(\underline{\mathfrak{g}}\subset \ker\omega_0\).
\end{itemize}
The following nonembeddable example shows that these conditions are important.
\begin{example}
	Let \(X_1=S^3\) be the 3-sphere together with a CR structure \(T^{1,0}X_1\) such that \((X_1,T^{1,0}X_1)\) is not realizable as CR submanifold of the euclidean space (see \cite{Burns02} and also \cite{Burns}, \cite{Jacobowitz}, \cite{Rossi}) and let \((X_2,T^{1,0}X_2)\) be a strongly pseudoconvex CR manifold of codimension one with a transversal CR \(S^1\)-action. Consider the CR manifold \((X,T^{1,0}X)\) of codimension two given by \(X=X_1\times X_2\) and \(T^{1,0}X=T^{1,0}X_1\oplus T^{1,0}X_2\).
	We have that \(X\) admits a CR \(S^1\) action. But \(X\) is not CR embeddable into some \(\C^N\) since \(X_1\) is not CR embeddable.
	Let us see which of the previous assumptions fails to be satisfied.
	We can choose a non-vanishing real one form \(\omega_1\in\Omega^1(X_1)\) with \(\omega_1(T^{1,0}X_1)=0\) and let \(T_1\) be a vector field with \(\omega_1(T_1)=-1\). Let \(T_2\) be the vector field induced by the transversal CR \(S^1\) action on \(X_2\) and \(\omega_2\) the unique real one form defined by \(\omega_2(T_2)=-1\) and \(\omega_2(T^{1,0}X_2)=0\).
	We can identify \(T_1,T_2,\omega_1,\omega_2\) with vector fields and one forms on \(X=X_1\times X_2\) in a natural way. Let \(\omega_0\in\Omega^1(X)\) be a real one form with \(T^{1,0}X\oplus T^{0,1}X\subset \ker \omega_0\). Since \(\omega_0(T^{1,0}X)=0\) we can write \(\omega_0=a\omega_1+b\omega_2\) for smooth functions \(a,b\in C^\infty (X,\R)\). Assuming positivity of \(-\frac{1}{2i}d\omega_0\) (Assumption~(i)) we find that \(a,b>0\) must hold. The assumption \(\omega_0(T_2)=0\) (Assumption~(ii)) leads to \(b=0\). This shows that both assumptions cannot be satisfied at the same time in this example.
\end{example} 

In Section~\ref{s-gue181015}, we introduce the notion of CR orbifolds and study some basic properties of CR orbifolds. In Section~\ref{s-gue180920}, we establish a CR orbifold version of Boutet de Monvel's embedding theorem. 
\begin{theorem}\label{t-gue180922m}
Under the assumptions used in Theorem~\ref{t-gue180803m}, the CR orbifold $X/G$ can be CR embedded into \(\C^N\), for some \(N\in\N\). 
\end{theorem}
It turns out that any effective CR orbifold can be written as a quotient \(X/G\) where \(X\) is a CR manifold equipped with a CR action of a compact Lie group \(G\) (see Lemma~\ref{t-gue181015}). 
\begin{corollary}\label{thm:BMOrbifold}
	Let \(X\) be a compact and orientable strongly pseudoconvex effective CR orbifold of codimension one. Then \(X\) can be CR embedded  into \(\C^N\), for some \(N\in\N\).
\end{corollary}
 
 This paper is organized as follows. In Section~\ref{s:prelim}, we fix some terminology and give basic definitions and examples for CR manifolds of high codimension with Lie group actions. In Section~\ref{s-gue180710m}, we study the Fourier decomposition of the Kohn Laplacian. Section~\ref{s-gue180722} contains the proof of Theorem~\ref{t-gue180803m}.
 The orbifold version of Boutet de Monvel's embedding theorem (see Corollary~\ref{thm:BMOrbifold}) is proven in Section~\ref{Sec:CROrbifolds}. Combining the results of Section~\ref{s-gue180722} and Section~\ref{Sec:CROrbifolds}, we prove Theorem~\ref{thm:CREmbedding} in Section~\ref{Sec:InducedCRStructures}.

\section{Preliminaries}\label{s:prelim}

\subsection{Some standard notations}\label{s-gue150508b}
We use the following notations: $\mathbb N=\set{1,2,\ldots}$,
$\mathbb N_0=\mathbb N\cup\set{0}$, $\Real$
is the set of real numbers,
\[\Real_+:=\set{x\in\Real;\, x>0},\ \ \ol\Real_+:=\set{x\in\Real;\, x\geq0}.\]
For a multiindex $\alpha=(\alpha_1,\ldots,\alpha_m)\in\mathbb N_0^m$
we set $\abs{\alpha}=\alpha_1+\cdots+\alpha_m$. For $x=(x_1,\ldots,x_m)\in\Real^m$ we write
\[
x^\alpha=x_1^{\alpha_1}\ldots x^{\alpha_m}_m,\quad
 \pr_{x_j}=\frac{\pr}{\pr x_j}\,,\quad
\pr^\alpha_x=\pr^{\alpha_1}_{x_1}\ldots\pr^{\alpha_m}_{x_m}=\frac{\pr^{\abs{\alpha}}}{\pr x^\alpha}\,.
\]
Let $z=(z_1,\ldots,z_m)$, $z_j=x_{2j-1}+ix_{2j}$, $j=1,\ldots,m$, be coordinates of $\Complex^m$,
where
$x=(x_1,\ldots,x_{2m})\in\Real^{2m}$ are coordinates of $\Real^{2m}$.
Throughout the paper we also use the notation
$w=(w_1,\ldots,w_m)\in\Complex^m$, $w_j=y_{2j-1}+iy_{2j}$, $j=1,\ldots,m$, where
$y=(y_1,\ldots,y_{2m})\in\Real^{2m}$.
We write
\[
\begin{split}
&z^\alpha=z_1^{\alpha_1}\ldots z^{\alpha_m}_m\,,\quad\ol z^\alpha=\ol z_1^{\alpha_1}\ldots\ol z^{\alpha_m}_m\,,\\
&\pr_{z_j}=\frac{\pr}{\pr z_j}=
\frac{1}{2}\Big(\frac{\pr}{\pr x_{2j-1}}-i\frac{\pr}{\pr x_{2j}}\Big)\,,\quad\pr_{\ol z_j}=
\frac{\pr}{\pr\ol z_j}=\frac{1}{2}\Big(\frac{\pr}{\pr x_{2j-1}}+i\frac{\pr}{\pr x_{2j}}\Big),\\
&\pr^\alpha_z=\pr^{\alpha_1}_{z_1}\ldots\pr^{\alpha_m}_{z_m}=\frac{\pr^{\abs{\alpha}}}{\pr z^\alpha}\,,\quad
\pr^\alpha_{\ol z}=\pr^{\alpha_1}_{\ol z_1}\ldots\pr^{\alpha_m}_{\ol z_m}=
\frac{\pr^{\abs{\alpha}}}{\pr\ol z^\alpha}\,.
\end{split}
\]

Let $X$ be a $C^\infty$ orientable paracompact manifold.
We let $TX$ and $T^*X$ denote the tangent bundle of $X$ and the cotangent bundle of $X$ respectively.
The complexified tangent bundle of $X$ and the complexified cotangent bundle of $X$
will be denoted by $\Complex TX$ and $\Complex T^*X$ respectively. We write $\langle\,\cdot\,,\cdot\,\rangle$
to denote the pointwise duality between $TX$ and $T^*X$.
We extend $\langle\,\cdot\,,\cdot\,\rangle$ bilinearly to $\Complex TX\times\Complex T^*X$.

Let $E$ be a $C^\infty$ vector bundle over $X$. The fiber of $E$ at $x\in X$ will be denoted by $E_x$.
Let $F$ be another vector bundle over $X$. We write
$F\boxtimes E^*$ to denote the vector bundle over $X\times X$ with fiber over $(x, y)\in X\times X$
consisting of the linear maps from $E_y$ to $F_x$.

Let $Y\subset X$ be an open set. The spaces of
smooth sections of $E$ over $Y$ and distribution sections of $E$ over $Y$ will be denoted by $C^\infty(Y, E)$ and $\mathscr D'(Y, E)$ respectively.
Let $\mathscr E'(Y, E)$ be the subspace of $\mathscr D'(Y, E)$ whose elements have compact support in $Y$.
For $m\in\Real$, we let $H^m(Y, E)$ denote the Sobolev space
of order $m$ of sections of $E$ over $Y$. Put
\begin{gather*}
H^m_{\rm loc\,}(Y, E)=\big\{u\in\mathscr D'(Y, E);\, \varphi u\in H^m(Y, E),
      \,\forall\varphi\in C^\infty_0(Y)\big\}\,,\\
      H^m_{\rm comp\,}(Y, E)=H^m_{\rm loc}(Y, E)\cap\mathscr E'(Y, E)\,.
\end{gather*}

\subsection{CR manifolds with high codimension}\label{s-gue180710I} 

Let $(X, T^{1,0}X)$ be a compact and orientable CR manifold of dimension $2n+d+1$, $n\geq 2$, $d\geq1$, where $T^{1,0}X$ is a CR structure of $X$, that is, $T^{1,0}X$ is a subbundle of rank $n$ of the complexified tangent bundle $\mathbb{C}TX$, satisfying $T^{1,0}X\cap T^{0,1}X=\{0\}$, where $T^{0,1}X=\overline{T^{1,0}X}$, and $[\mathcal V,\mathcal V]\subset\mathcal V$, where $\mathcal V=C^\infty(X, T^{1,0}X)$. In this work, we assume that $X$ admits a action of a $d$-dimensional compact Lie group $G$.  Let $\mathfrak{g}$ denote the Lie algebra of $G$. For any $\xi \in \mathfrak{g}$, we write $\xi_X$ to denote the vector field on $X$ induced by $\xi$. That is, $(\xi_X u)(x)=\frac{\partial}{\partial t}\left(u(e^{t\xi}\circ x)\right)|_{t=0}$, for any $u\in C^\infty(X)$. Let $\underline{\mathfrak{g}}={\rm Span\,}(\xi_X;\, \xi\in\mathfrak{g})$. 
\begin{definition}\label{d-gue180710}
We say that the Lie group action of $G$ is CR if for every $\xi_X\in\underline{\mathfrak{g}}$, we have 
\[[\xi_X, \mathcal V]\subset\mathcal V,\]
where $\mathcal V=C^\infty(X, T^{1,0}X)$. 
\end{definition}

We assume throughout that  the action of $G$ is CR. 
Let $\hat T\in C^\infty(X,TX)$ be a global 
defined vector field such that 
 \begin{equation}\label{e-gue180710}
 \Complex TX=T^{1,0}X\oplus T^{0,1}X\oplus\Complex\hat T\oplus\Complex\underline{\mathfrak{g}}.
 \end{equation} 
 Note that under the assumption \((T^{1,0}X\oplus T^{0,1}X)\cap \Complex\underline{\mathfrak{g}}=\{0\}\) we have that such a vector field \(\hat{T}\) exists if and only if \(X\) is orientable. 
Let $\hat\omega_0(x)\in C^\infty(X,T^*X)$ be the globally real one form on $X$ such that 
\begin{equation}\label{e-gue180710I}
\begin{split}
&\langle\,\hat\omega_0\,,\,V\,\rangle=0,\ \ \forall V\in T^{1,0}X\oplus T^{0,1}X\oplus\Complex\underline{\mathfrak{g}}, \\
&\langle\,\hat\omega_0\,,\,\hat T\,\rangle=-1\ \ \mbox{on $X$}.
\end{split}
\end{equation}

\begin{definition}\label{d-gue180710I}
For $p\in X$, the Levi form with respect to $\hat\omega_0$ at $p$ is the Hermitian quadratic form on $T^{1,0}_pX$ given by 
\[\mathcal{L}_{\hat\omega_0,p}(U,V):=-\frac{1}{2i}\langle\,d\hat\omega_0(p)\,,\,U\wedge\ol V\,\rangle,\ \ \forall U, V\in T^{1,0}_pX.\]
\end{definition}
In this work, we assume that $\mathcal{L}_{\hat\omega_0,x}$ is positive definite at every point $x\in X$. We refer the reader to Section~\ref{Sec:Examples} for examples of CR manifolds which satisfy the conditions above.

Fix $g\in G$. Let $g^*:\Lambda^r_x(\Complex T^*X)\To\Lambda^r_{g^{-1}\circ x}(\Complex T^*X)$ be the pull-back map.
Fix a Haar measure $d\mu=d\mu(g)$ on $G$ so that $\int_Gd\mu(g)=1$. Put 
\begin{equation}\label{e-gue180712}
\omega_0(x):=\int_G(g^*\hat\omega_0)(x)d\mu(g)\in C^\infty(X,T^*X).
\end{equation}
That is, $\omega_0(x)$ is the global one form on $X$ defined as follows: For every $x\in X$ and every $V\in T_xX$, we have
\[\langle\,\omega_0(x)\,,\,V\,\rangle=\int_G\langle\,(g^*\hat\omega_0)(x)\,,\,V\,\rangle d\mu(g)
=\int_G\langle\,\hat\omega_0(g\circ x)\,,\,(dg)V\,\rangle d\mu(g).\]
Then $\omega_0(x)$ is a $G$-invariant global one form.
\begin{lemma}\label{l-gue180712}
We have that $\omega_0(x)$ is a non-vanishing global one form on $X$, 
\[\langle\,\omega_0\,,\,V\,\rangle=0,\ \ \forall V\in T^{1,0}X\oplus T^{0,1}X\oplus\Complex\underline{\mathfrak{g}}\]
and $\mathcal{L}_{\omega_0,x}$ the Levi form with respect to $\omega_0$ is positive definite at every point $x\in X$. 
\end{lemma}

\begin{proof}
Fix any $G$-invariant Hermitian metric $\langle\,\cdot\,,\,\cdot\,\rangle$ on $\Complex TX$. Since $\mathcal{L}_{\hat\omega_0,x}$ is positive definite on $X$ and $X$ is compact, there is a constant $C>0$ such that 
\begin{equation}\label{e-gue180712m}
-\frac{1}{2i}\langle\,d\hat\omega_0(x)\,,\,U\wedge\ol U\,\rangle\geq C,\ \ \mbox{for every  $U\in T^{1,0}_xX$ with $\langle\,U\,,\,U\,\rangle=1$ and every $x\in X$}.
\end{equation}
Now, for every $U\in T^{1,0}_xX$ with $\langle\,U\,,\,U\,\rangle=1$ and every $x\in X$, we have 
\begin{equation}\label{e-gue180712mI}
\begin{split}
&-\frac{1}{2i}\langle\,d\omega_0(x)\,,\,U\wedge\ol U\,\rangle=-\frac{1}{2i}\int_G\langle\,d(g^*\hat\omega_0)(x)\,,\,U\wedge\ol U\,\rangle d\mu(g)\\
&=-\frac{1}{2i}\int_G\langle\,g^*(d\hat\omega_0)(x)\,,\,U\wedge\ol U\,\rangle d\mu(g)=-\frac{1}{2i}\int_G\langle\,d\hat\omega_0)(g\circ x)\,,\,(dg)U\wedge \ol{dg(U)}\,\rangle d\mu(g).
 \end{split}
\end{equation}
Since $G$ is CR and the Hermitian metric $\langle\,\cdot\,,\,\cdot\,\rangle$ is  $G$-invariant, we have $(dg)U\in T^{1,0}_{g\circ x}X$ and 
$\langle\,(dg)U\,,\,(dg)U\,\rangle=1$, for every $g\in G$. From this observation, \eqref{e-gue180712m} and \eqref{e-gue180712mI}, we deduce that 
\[-\frac{1}{2i}\langle\,d\omega_0(x)\,,\,U\wedge\ol U\,\rangle\geq C,\ \ \mbox{for every  $U\in T^{1,0}_xX$ with $\langle\,U\,,\,U\,\rangle=1$ and every $x\in X$},\]
where $C>0$ is the constant as in \eqref{e-gue180712m}. Hence, $\omega_0(x)$ is a non-vanishing global one form on $X$ and $\mathcal{L}_{\omega_0,x}$ is positive definite at every point $x\in X$. 

For every $V\in T^{1,0}_xX\oplus T^{0,1}_xX\oplus\Complex\underline{\mathfrak{g}}_x$, we have 
\[\begin{split}
\langle\,\omega_0(x)\,,\,V\,\rangle=\int_G\langle\,(g^*\hat\omega_0)(x)\,,\,V\,\rangle d\mu(g)=\int_G\langle\,\hat\omega_0(g\circ x)\,,\,(dg)V\,\rangle d\mu(g)=0
\end{split}\]
since $(dg)V\in  T^{1,0}_{g\circ x}X\oplus T^{0,1}_{g\circ x}X\oplus\Complex\underline{\mathfrak{g}}_{g\circ x}$, for every $g\in G$. The lemma follows. 
\end{proof}

Fix any global one form $\Td T\in C^\infty(X, TX)$ with $\langle\,\omega_0\,,\,\Td T\,\rangle=-1$ on $X$. Put 
\[T(x):=\int_G (g^*\Td T)(x)d\mu(g),\]
where $g^*\Td T$ denotes the pull-back of $\Td T$. Recall that $(g^*\Td T)(x)=dg^{-1}(\Td T(g\circ x))$, where $dg^{-1}: T_{g\circ x}X\To T_xX$ is the differential of the map $g^{-1}: X\To X$, $x\To g^{-1}\circ x$. 
Then, $T$ is a $G$-invariant global vector filed on $X$.

\begin{lemma}\label{l-gue180712I}
We have $\langle\,\omega_0\,,\,T\,\rangle=-1$ on $X$.
\end{lemma}

\begin{proof}
For every $x\in X$, we have 
\[\begin{split}
\langle\,\omega_0(x)\,,\,T(x)\,\rangle&=\int_G\langle\,\omega_0(x)\,,\,(dg^{-1})\Td T(g\circ x)\,\rangle d\mu(g)\\
&=\int_G\langle\,((g^{-1})^*\omega_0)(g\circ x)\,,\,\Td T(g\circ x)\,\rangle d\mu(g)\\
&=\int_G\langle\,\omega_0(g\circ x)\,,\,\Td T(g\circ x)\,\rangle d\mu(g)=-1
\end{split}\]
since $(g^{-1})^*\omega_0=\omega$, for every $g\in G$. The lemma follows. 
\end{proof}

From Lemma~\ref{l-gue180712I}, we deduce that there is a $G$-invariant global 
defined vector field $T\in C^\infty(X, TX)$ such that 
 \begin{equation}\label{e-gue180710z}
 \Complex TX=T^{1,0}X\oplus T^{0,1}X\oplus\Complex T\oplus\Complex\underline{\mathfrak{g}}
 \end{equation} 
 and there is a $G$-invariant global one form $\omega_0\in C^\infty(X,T^*X)$ such that 
 \begin{equation}\label{e-gue180712pa}
\begin{split}
&\langle\,\omega_0\,,\,V\,\rangle=0,\ \ \forall V\in T^{1,0}X\oplus T^{0,1}X\oplus\Complex\underline{\mathfrak{g}}, \\
&\langle\,\omega_0\,,\,T\,\rangle=-1\ \ \mbox{on $X$},\\
&\mbox{$\mathcal{L}_{\omega_0,x}$ the Levi form with respect to $\omega_0$ is positive definite at every point $x\in X$}.
\end{split}
\end{equation}

Denote by $T^{*1,0} X$ and $T^{*0,1}X$ the dual bundles of
$T^{1,0}X$ and $T^{0,1}X$ respectively. That is, 
\[\begin{split}
&T^{*1,0}X=\Bigr(T^{0,1}X\oplus\Complex T\oplus\Complex\underline{\mathfrak{g}}\Bigr)^\perp,\\
&T^{*0,1}X=\Bigr(T^{1,0}X\oplus\Complex T\oplus\Complex\underline{\mathfrak{g}}\Bigr)^\perp.
\end{split}\]
Define the vector bundle of $(0,q)$ forms by
$T^{*0,q}X:=\Lambda^q(T^{*0,1}X)$.
Let $D\subset X$ be an open set. Let $\Omega^{0,q}(D)$
denote the space of smooth sections of $T^{*0,q}X$ over $D$ and let $\Omega_0^{0,q}(D)$
be the subspace of $\Omega^{0,q}(D)$ whose elements have compact support in $D$. We write $C^\infty(D):=\Omega^{0,0}(D)$, $C^\infty_0(D):=\Omega^{0,0}_0(D)$. 

Take any $G$-invariant Hermitian metric on $TG$ and let $\xi_1,\ldots,\xi_d$ be an orthonormal basis for $\mathfrak{g}$. Put 
\begin{equation}\label{e-gue180713}
T_j:=\xi_{j,X}\in C^\infty(X,TX),\ \ j=1,\ldots,d,
\end{equation}
where $\xi_{j,X}$ denotes the vector field on $X$ induced by $\xi_j$, $j=1,\ldots,d$. 
Fix a $G$-invariant Hermitian metric $\langle\,\cdot\,|\,\cdot\,\rangle$ on $\Complex TX$ such that 
\begin{equation}\label{e-gue180710II}
\begin{split}
&T^{1,0}X\perp T^{0,1}X\perp\Complex T\perp\underline{\mathfrak{g}},\\
&\langle\,T\,|\,T\,\rangle=1.
\end{split}
\end{equation}
The $G$-invariant Hermitian metric $\langle\,\cdot\,|\,\cdot\,\rangle$ on $\Complex TX$ induces a $G$-invariant Hermitian metric $\langle\,\cdot\,|\,\cdot\,\rangle$ on the bundle $\oplus^{2n+d+1}_{j=1}\Lambda^j(\Complex T^*X)$ and let $\abs{\cdot}$ denote the corresponding norm. For $q=1,\ldots,n$, let
\[\tau^{0,q}: \Lambda^q(\Complex T^*X)\To T^{*0,q}X\]
be the orthogonal projection with respect to $\langle\,\cdot\,|\,\cdot\,\rangle$. The tangential Cauchy Riemann operator is given by
\begin{equation}\label{e-gue180710m}
\ddbar_b:=\tau^{0,q+1}\circ d: \Omega^{0,q}(X)\To\Omega^{0,q+1}(X).
\end{equation}
Let $dv_X=dv_X(x)$ be the volume form on $X$ induced by the Hermitian metric $\langle\,\cdot\,|\,\cdot\,\rangle$.
The natural global $L^2$ inner product $(\,\cdot\,|\,\cdot\,)$ on $C^\infty(X,\Lambda^r(\Complex T^*X))$ 
induced by $dv_X(x)$ and $\langle\,\cdot\,|\,\cdot\,\rangle$ is given by
\begin{equation}\label{e-gue180710mI}
(\,u\,|\,v\,):=\int_X\langle\,u(x)\,|\,v(x)\,\rangle\, dv_X(x)\,,\quad u,v\in C^\infty(X,\Lambda^r(\Complex T^*X))\,.
\end{equation} 
For $u\in C^\infty(X,\Lambda^r(\Complex T^*X))$, we write $\norm{u}^2:=(\,u\,|\,u\,)$. Let $L^2(X, \Lambda^r(\Complex T^*X))$ and $L^2_{(0,q)}(X)$ be the completions of $C^\infty(X,\Lambda^r(\Complex T^*X))$ and $\Omega^{0,q}(X)$ with respect to 
$(\,\cdot\,|\,\cdot\,)$ respectively. We write $L^2(X):=L^2_{(0,0)}(X)$. We extend $\ddbar_b$ to $L^2_{(0,q)}(X)$ by 
\[\begin{split}
&\ddbar_b:{\rm Dom\,}\ddbar_b\subset L^2_{(0,q)}(X)\To L^2_{(0,q+1)}(X),\\ 
&{\rm Dom\,}\ddbar_b:=\set{u\in L^2(X);\, \ddbar_bu\in L^2_{(0,1)}(X)}.
\end{split}\]
Put 
\[{\rm Ker\,}\ddbar_b:=\set{u\in L^2(X);\, \ddbar_bu=0}.\]

Since the action of $G$ is CR, we have 
\[g^*:T^{*0,q}_xX\To T^{*0,q}_{g^{-1}\circ x}X,\  \ \forall x\in X.\]
Thus, for $u\in\Omega^{0,q}(X)$, we have $g^*u\in\Omega^{0,q}(X)$. Moreover, we have 
\begin{equation}\label{e-gue180711p}
\ddbar_b(g^*u)=g^*(\ddbar_bu),\ \ \forall u\in\Omega^{0,q}(X),\ \ \forall g\in G. 
\end{equation}
Let $\ol{\pr}^*_b: \Omega^{0,q+1}(X)\To\Omega^{0,q}(X)$ be the formal adjoint of $\ddbar_b$ with respect to $(\,\cdot\,|\,\cdot\,)$ and let 
\[\Box^{(q)}_b:=\ddbar_b\,\ol{\pr}^*_b+\ol{\pr}^*_b\,\ddbar_b: \Omega^{0,q}(X)\To\Omega^{0,q}(X).\]
We  can check that 
\begin{equation}\label{e-gue180711pI}
\begin{split}
\ol{\pr}^*_b(g^*u)=g^*(\ol{\pr}^*_bu),\ \ \forall u\in\Omega^{0,q}(X),\ \ \forall g\in G,\\
\Box^{(q)}_b(g^*u)=g^*(\Box^{(q)}_bu),\ \ \forall u\in\Omega^{0,q}(X),\ \ \forall g\in G. 
\end{split}
\end{equation}

Let $d: C^\infty(X, \Lambda^r(\Complex T^*X))\To  C^\infty(X, \Lambda^{r+1}(\Complex T^*X))$ be the exterior derivative and let 
$d^*: C^\infty(X, \Lambda^{r+1}(\Complex T^*X))\To  C^\infty(X, \Lambda^{r}(\Complex T^*X))$ be the formal adjoint of $d$ with respect to $(\,\cdot\,|\,\cdot\,)$. 
Let 
\begin{equation}\label{e-gue180716}
\triangle^{(r)}:=d^*d+dd^*: C^\infty(X, \Lambda^r(\Complex T^*X))\To C^\infty(X, \Lambda^r(\Complex T^*X)).
\end{equation}
We have 
\begin{equation}\label{e-gue180711pII}
\triangle^{(r)}(g^*u)=g^*(\triangle^{(r)}u),\ \ \forall u\in C^\infty(X, \Lambda^r(\Complex T^*X)),\ \ \forall g\in G.
\end{equation}

Consider $\triangle^{(r)}+I: C^\infty(X, \Lambda^r(\Complex T^*X))\To C^\infty(X, \Lambda^r(\Complex T^*X))$. We extend $\triangle^{(r)}+I$ to the $L^2$ space by 
\[\begin{split}
&\triangle^{(r)}+I: {\rm Dom\,}(\triangle^{(r)}+I)\subset L^2(X,\Lambda^r(\Complex T^*X))\To L^2(X,\Lambda^r(\Complex T^*X)),\\
&{\rm Dom\,}(\triangle^{(r)}+I)=\set{u\in L^2(X,\Lambda^r(\Complex T^*X));\, (\triangle^{(r)}+I)u\in L^2(X,\Lambda^r(\Complex T^*X))}.
\end{split}\]
The operator $\triangle^{(r)}+I$ is a nonnegative self-adjoint operator. Let ${\rm Spec\,}(\triangle^{(r)}+I)$ denote the spectrum of $\triangle^{(r)}+I$. Then, ${\rm Spec\,}(\triangle^{(r)}+I)$ is a discrete subset of $]0,+\infty[$ and for every $\lambda\in {\rm Spec\,}(\triangle^{(r)}+I)$, $\lambda$ is an eigenvalue of $\triangle^{(r)}+I$ and the space 
\[E_\lambda(X,\Lambda^r(\Complex T^*X)):=\set{u\in{\rm Dom\,}(\triangle^{(r)}+I);\, (\triangle^{(r)}+I)u=\lambda u}\]
is a finte dimensional subspace of $C^\infty(X, \Lambda^r(\Complex T^*X))$. For every $\lambda\in {\rm Spec\,}(\triangle^{(r)}+I)$, let 
\[P_\lambda: L^2(X,\Lambda^r(\Complex T^*X))\To E_\lambda(X,\Lambda^r(\Complex T^*X))\]
be the orthogonal projection with respect to $(\,\cdot\,|\,\cdot\,)$. The square root of $\triangle^{(r)}+I$ is given by 
\begin{equation}\label{e-gue180719}
\begin{split}
\sqrt{\triangle^{(r)}+I}: C^\infty(X, \Lambda^r(\Complex T^*X))&\To L^2(X,\Lambda^r(\Complex T^*X)),\\
u&\To \sum_{\lambda\in{\rm Spec\,}(\triangle^{(r)}+I)}\sqrt{\lambda}(P_\lambda u).
\end{split}
\end{equation}
It is easy to see that \eqref{e-gue180719} is well-defined. 

We extend $\sqrt{\triangle^{(r)}+I}$ to the $L^2$ space by 
\[\begin{split}
&\sqrt{\triangle^{(r)}+I}: {\rm Dom\,}\sqrt{\triangle^{(r)}+I}\subset L^2(X,\Lambda^r(\Complex T^*X))\To L^2(X,\Lambda^r(\Complex T^*X)),\\
&{\rm Dom\,}\sqrt{\triangle^{(r)}+I}=\set{u\in L^2(X,\Lambda^r(\Complex T^*X));\, \sqrt{\triangle^{(r)}+I}u\in L^2(X,\Lambda^r(\Complex T^*X))}.
\end{split}\]
It is well-known (see~\cite{Seeley}) that
\begin{equation}\label{e-gue180719I}
\mbox{$\sqrt{\triangle^{(r)}+I}$ is a classical elliptic pseudodifferential operator of order one}.
\end{equation}
Hence, $\sqrt{\triangle^{(r)}+I}: C^\infty(X, \Lambda^r(\Complex T^*X))\To C^\infty(X, \Lambda^r(\Complex T^*X))$. 
Moreover, from \eqref{e-gue180711pII}, it is not difficult to check that 
\begin{equation}\label{e-gue180719II}
\sqrt{\triangle^{(r)}+I}(g^*u)=g^*(\sqrt{\triangle^{(r)}+I}u),\ \ \forall u\in C^\infty(X, \Lambda^r(\Complex T^*X)),\ \ \forall g\in G.
\end{equation}

\subsection{Examples}\label{Sec:Examples}
	In this section we state some examples of CR manifolds with high codimension in the context of  Section~\ref{s-gue180710I}. 
\begin{example}\label{HighCodimensionExample1}
	Let $(Z,\omega)$ be a K\"ahler manifold. Let $G$ be a compact Lie group acting holomorphic on $Z$ and leaving $\omega$ invariant. 
	Assume that there exists a moment map $\mu \colon Z \rightarrow \mathfrak{g}^*$ and $0 \in \mathfrak{g}^*$ is a regular value. 
	We will show that $\mathcal{M} := \mu^{-1}(0)$ is a CR submanifold of $Z$ with transversal $G$-action.\\
	First note that $\mathcal{M}$ is a $G$-invariant manifold, since $\mu$ is equivariant.
	Then we have $T_x \mathcal{M} = $ ker $d_x \mu = (\mathfrak{g} x )^{\perp_\omega} = \mathfrak{g} x \oplus (\Complex \mathfrak{g}x)^{\perp_\omega}$.
	Because $0$ is a regular value, we have that the dimension of $\mathfrak{g}x$ does not depend on $x \in \mathcal{M}$ and $\mathfrak{g}x$ can not contain a complex subspace because of $\mathfrak{g}x \subset (\mathfrak{g}x)^{\perp_\omega}$.
\end{example}

Note that momentum maps and the zero levels of momentum maps play an important role in the study of Lie group actions. 
Especially the quotient $\mathcal{M} / G$ is of high interest.

\begin{example}
	Let $X$ be a CR manifold of dimension $2n+d$ and CR codimension $d$ with transversal CR action of a compact group $G$. 
	Let $U$ be an open, $G$-invariant subset of $X$ with smooth boundary.
	Then $Y := \partial U$ is a CR submanifold of $X$ of dimension $2n+d-1= 2n-2 +d+1$ with CR codimension $d+1$.\\
	To see this, one checks that $\Complex TY =\Complex \underline{\mathfrak{g}} \oplus (T^{1,0} X \oplus T^{0,1} X) \cap \Complex TY$ and dim$_\Complex T^{1,0}X \cap \Complex TY = n-1$.\\
	Now assume that $X$ is orientable and $U$ is given by $U = \{ \rho <0 \}$, where $\rho$ is a $G$-invariant strictly plurisubharmonic function on $X$. By strictly plurisubharmonic we mean that $\langle\,\pr_b\,\ddbar_b\rho(x)\,,\,Z\wedge\ol Z\,\rangle>0$, for every $x\in X$ and every $Z\in T^{1,0}_xX$, $Z\neq0$, where 
	the differential operators are defined in (\ref{e-gue180710m}).
	We may define a one-form $\omega$ on $Y$ via $\omega = (-\partial_b+\ddbar_b)\rho$ on $(T^{1,0} X \oplus T^{0,1}X) \cap \C TY$ and $\omega( \underline{\mathfrak{g}}) = 0$.
	We then have $\omega( T^{1,0} Y \oplus T^{0,1} Y) = 0$, $\omega (T) \neq 0$ for some non-vanishing vector field $T$ and one checks in local coordinates that $d \omega$ is positive on $T^{1,0} Y$.
\end{example}

\begin{example}
	Let $X$ be a CR manifold of codimension $1$ with a CR group action of a compact group $G$ such that $\underline{\mathfrak{g}} \subset T^{1,0}X \oplus T^{0,1}X$.
	Assume that there exists a $G$-invariant $1$-form $\omega$ such that $(U,V) \mapsto d\omega (U, \bar{V})$ is a positive definite form on $T^{1,0}X$.
	Furthermore, assume that there exists a vector field $T$ on $X$ such that $d\omega (T, \xi) = 0$ for $\xi \in \underline{\mathfrak{g}}$ and $\Complex T \oplus T^{1,0}X \oplus T^{0,1}X = \Complex TX$.\\
	Define $\mu^\xi (x) := -\omega( \xi_X (x))$ for $\xi \in \mathfrak{u}$ and $\xi_X$ the induced vector field on $X$.\\
	Since $\omega$ is $G$-invariant, we have $d \mu^\xi = -d \iota_{\xi_X} \omega = \iota_{\xi_X} d\omega$.
	Define $\mu \colon X \rightarrow \mathfrak{g}^*$, $\mu(x)(\xi) = \mu^\xi (x)$ and assume that $0$ is a regular value.\\
	We see that the manifold $\mathcal{M} := \mu^{-1}(0)$ is $G$-invariant because $\mu^\xi (gx) = \omega( \xi(gx)) = \omega (d g ( Ad_{g^{-1}} \xi)(x)) = \omega ((Ad_{g^{-1}} \xi)(x) ) = 0$ for $x \in \mathcal{M}$ and $Ad$ denoting the adjoint action of $G$ on $\mathfrak{g}$.
	
	We have $d \mu^\xi (T) = d\omega( \xi_X, T) = 0$. 
	The CR structure on $X$ induces a subbundle $W:=(T^{1,0}X \oplus T^{0,1}X) \cap TX$ and a linear map $J_x \colon W_x \rightarrow W_x$ with $J_x^2 =Id_x$.
	The form $d\omega$ induces hermitian metrics on the complex spaces $(W_x, J_x)$ and one checks that $T\mathcal{M} = T \oplus \underline{\mathfrak{g}}^{\perp_{d\omega}} = T \oplus \underline{\mathfrak{g}} \oplus (\Complex_J \underline{\mathfrak{g}})^{\perp_{d \omega}}$.
	As in example \ref{HighCodimensionExample1}, the bundle $\underline{\mathfrak{g}}$ does not contain a complex subspace of $W$ and therefore, $\mathcal{M}$ is a CR submanifold of $X$.
\end{example}

\begin{example}
	Let \(X=\{(z,z_{n+1})\in\C^{n}\times \C\mid \|z\|^2-|z_{n+1}|^2=1\}\). Putting \(T^{1,0}X=\C TX\cap T^{1,0}\C^{n+1}\) we have that \((X,T^{1,0}X)\) is a CR manifold of codimension one which is neither compact nor strongly pseudoconvex. We have that \[\omega=\frac{i}{2}\left(z_{n+1}d\overline{z_{n+1}}-\overline{z_{n+1}}dz_{n+1}+\sum_{j=1}^n \overline{z}_jdz_j-z_jd\overline{z}_j\right) \]
	defines a nonvanishing real one form on \(X\) with \(T^{1,0}X\oplus T^{0,1}X \subset \ker\omega\). 
	Furthermore, \(T:=i\sum_{j=1}^{n+1}z_j\frac{\partial}{\partial z_j}-\overline{z}_j\frac{\partial}{\partial \overline{z}_j}\) is a transversal vector field with \(\omega(T)=-1\). Consider the CR \(S^1\) action on \(X\) given by \(s\circ(z,z')=(sz,s^2z')\) and 
	denote its induced vector field by \(T_1\). Since \(\omega\) is invariant under the \(S^1\) action a CR moment map is given by \(\mu\colon X\to \R\), \(\mu(z,z')=\langle\omega,T_1\rangle(z,z')=1-|z_{n+1}|^2 \). It follows that \((\mathcal{M}:=\mu^{-1}(0),T^{1,0}\mathcal{M}:=\C T\mathcal{M}\cap T^{1,0}X)\) is a compact CR manifold of codimension two with \(\C T\mathcal{M}= T^{1,0}\mathcal{M}\oplus T^{0,1}\mathcal{M}\oplus\C T_1\oplus\C T\). Denoting the pullback of \(\omega\) to \(\mathcal{M}\) by \(\omega_0\) we find \(\langle\omega_0,T\rangle=-1\), \(\langle\omega_0,T_1\rangle=0\) and \(T^{1,0}\mathcal{M}\oplus T^{0,1}\mathcal{M}\subset\ker\omega_0\). A direct calculation then shows  that \((\mathcal{M},T^{1,0}\mathcal{M})\) is strongly pseudoconvex in the direction of \(T\), that is, \(\mathcal{L}_{\omega_0}\)  (see Definition~\ref{d-gue180710I}) is positive everywhere.
\end{example}
\begin{example}\label{Ex:GluingAlongCenter}
	Let $X$ be a CR manifold of codimension $1$ and assume that there exists a transversal CR $S^1$-action on $X$. 
	We may embed $S^1$ into $U_n$, the unitary matrices of degree $n$, as multiples of the identity.
	Then $U_n \times X$ is a CR manifold on which $S^1$ acts via $(s, (U, x)) \mapsto (Us^{-1}, sx)$.
	The action is CR and free, therefore the quotient $(U_n \times X) /S^1 =: U_n \times^{S^1} X$ is a manifold and we may push the CR structure onto the quotient 
	$T^{1,0} (U_n \times^{S^1} X) := \pi_* T^{1,0} (U_n \times X)$, where $\pi$ denotes the projection.\\
	This does give a CR structure on $U_n \times^{S^1}X$, such that the $U_n$-action is CR, locally free and transversal.
	Note that we may decompose the $U_n$-action into an $S^1 \times SU_n$-action and if $X$ is a strictly pseudoconvex CR manifold, then $U_n \times^{S^1} X$ is positive in the sense of Definition \ref{d-gue180710I} for the vector field $T$ induced by the $S^1$-action.
\end{example}

\begin{example}
	Consider \(( U_n \times^{S^1} X,T^{1,0} (U_n \times^{S^1} X))\) as in Example~\ref{Ex:GluingAlongCenter}. We have that \(U_n \times^{S^1} X\) admits a transversal CR \(S^1\times SU_n\) action. Set \(\hat{X}=U_n \times^{S^1} X\) and denote by \(T\) the real vector field induced by the \(S^{1}\) part of the action. Take an \(SU_n\) invariant function \(\psi\in C^\infty(\hat{X},\R)\) and define a CR structure on \(\hat{X}\) by
	\[T^{1,0}_x\hat{X}=\{Z+i(T(\psi)Z-Z(\psi)T)\mid Z\in T^{1,0}_x (U_n \times^{S^1} X) \},\,\,\,x\in \hat{X}.\]
	We have that \((\hat{X},T^{1,0}\hat{X})\) admits a CR  \(SU_n\) action and that \(T\) is a transversal vector field, that is,
	\(\C T \hat{X}=T^{1,0}\hat{X}\oplus T^{0,1}\hat{X}\oplus \C\underline{\mathfrak{su}}_n \oplus \C T.\)
	Note that in general \(T\) fails to be a CR vector field for \((\hat{X},T^{1,0}\hat{X})\).
	Assuming that \(X\) is strongly pseudoconvex and that  the \(C^2\)-norm of \(\psi\) is sufficiently small we have that \((\hat{X},T^{1,0}\hat{X})\) is strongly pseudoconvex in the direction of \(T\). More precisely, there exists a real nonvanishing one form \(\omega_0\) with \(\langle\omega_0,T\rangle=-1\) and \(T^{1,0}\hat{X}\oplus T^{0,1}\hat{X}\oplus \C\underline{\mathfrak{su}}_n\subset\ker \omega_0\) such that \(\mathcal{L}_{\omega_0}\)  (see Definition~\ref{d-gue180710I}) is positive everywhere.
\end{example}

\begin{example}
	Let $(M, T^{1,0}M)$ be a compact connected orientable CR manifold of dimension $2n+1$ and codimension $1$, $n\geq1$. Fix a global non-vanishing real $1$-form $\omega_0\in C^\infty(M,T^*M)$ such that $\langle\,\omega_0\,,\,u\,\rangle=0$, for every $u\in T^{1,0}M\oplus T^{0,1}M$.  Assume that $M$ admits a $d$-dimensional connected compact CR Lie group action $G$ and $G$ preserves $\omega_0$. Let $\mathfrak{g}$ denote the Lie algebra of $G$. For any $\xi \in \mathfrak{g}$, we write $\xi_X$ to denote the vector field on $M$ induced by $\xi$. The moment map associated to the form $\omega_0$ is the map $\mu:M \to \mathfrak{g}^*$ such that, for all $x \in M$ and $\xi \in \mathfrak{g}$, we have $\langle \mu(x), \xi \rangle = \omega_0(\xi_X(x))$. Assume that $0$ is a regular value of $\mu$ and $-\frac{1}{2i}d\omega_0|_{T^{1,0}M}$ is positive at every point of $\mu^{-1}(0)$. It was shown in~\cite[Thm. 2.5]{HH17} that $\mu^{-1}(0)$ is a CR manifold of  dimension $2n+1-d$ and codimension $d+1$ and $\mu^{-1}(0)$ satisfies the assumptions of Theorem~\ref{t-gue180803m}. 
\end{example}

\section{Fourier analysis and the Kohn Laplacian}\label{s-gue180710m}

We first recall some classical facts about Fourier analysis on Lie groups in Section~\ref{s-gue180710mI}. Then, in Section~\ref{s-gue180712}, we give a description of \(G\)-finite functions on manifolds with Lie group action in terms of a Fourier decomposition. In Section~\ref{s-gue180716} we study the restriction of the Kohn Laplacian to the components of this Fourier decomposition.

\subsection{Fourier analysis on Lie groups}\label{s-gue180710mI}

We recall that a representation of the group $G$ is a group homomorphism $\rho: G\To GL(\Complex^d)$ for some $d\in\mathbb N$. The representation represents the elements of the group as $d\times d$ complex square matrices so that multiplication commutes with $\rho$. The number $d$ is the dimension of the representation $\rho$. 
A representation $\rho$ is unitary if each $\rho(g)$, $g\in G$, is a unitary matrix. A representation $\rho$ is reducible if we have a splitting $\Complex^d=V_1\oplus V_2$ so that $\rho(g)V_j =V_j$ for all $g\in G$, for both $j=1,2$ and $0<{\rm dim\,}V_1<d$, where $V_1$ and $V_2$ are vector subspaces of $\Complex^d$. If $\rho$ is not reducible, it is called irreducible. Two representations $\rho_1$ and $\rho_2$ are equivalent if they have the same dimension and there is an invertible matrix $A$ such that $\rho_1(g)=A\rho_2(g)A^{-1}$  for all $g\in G$. To understand all representations of the group $G$, it often suffices to study the irreducible unitary representations. 
We let 
\[\mathcal{R}=\set{\mathcal{R}_m;\, m=1,2,\ldots}\]
denote the set of all irreducible unitary representations of the group $G$, including only one representation from each equivalence class.
For each $\mathcal{R}_m$, we  write $\mathcal{R}_m$ as a matrix $\left(\mathcal{R}_{m,j,k}\right)^{d_m}_{j,k=1}$, where $d_m$ is the dimension of $\mathcal{R}_m$. Fix a Haar measure $d\mu(g)$ on $G$ so that $\int_Gd\mu(g)=1$ and let $(\,\cdot\,|\,\cdot\,)_G$ be the $L^2$ inner product on $C^\infty(G)$ induced by $d\mu(g)$. Let $L^2(G)$ be the completion of $C^\infty(G)$ with respect to  $(\,\cdot\,|\,\cdot\,)_G$. We recall the Peter-Weyl theorem on compact Lie groups (see~\cite{Taylor}).

\begin{theorem}\label{t-gue180711}
We have that the set $\set{\sqrt{d_m}\mathcal{R}_{m,j,k};\, j, k=1,\ldots,d_m, m=1,2,\ldots}$ form an
orthonormal basis of $L^2(G)$. 
\end{theorem}

Let $f\in C^\infty(G)$ be a smooth function. Fix an irreducible unitary representation $\mathcal{R}_m$ and fix a matrix element $\mathcal{R}_{m,j,k}$. The Fourier component of $f$ with respect to $\mathcal{R}_{m,j,k}$ is the function 
\begin{equation}\label{e-gue180710mp}
f_{m,j,k}(g):=d_m\mathcal{R}_{m,j,k}(g)\int_Gf(g)\ol{\mathcal{R}_{m,j,k}(g)}d\mu(g)\in C^\infty(G). 
\end{equation}
For every $\ell\in\mathbb N$, let $\norm{\cdot}_{C^\ell(G)}$ be a $C^\ell$ norm on $G$. The following result is the smooth version of the Peter-Weyl theorem on compact Lie groups (see  the discussion after Theorem 2 in~\cite{Taylor})

\begin{theorem}\label{t-gue180710mp}
With the notations used above, let $f\in C^\infty(G)$. Then, for every $\ell\in\mathbb N$ and every $\varepsilon>0$, there is a $N_0\in\mathbb N$ such that for every $N\geq N_0$, we have 
\begin{equation}\label{e-gue180710s}
\norm{f-\sum^N_{m=1}\sum^{d_m}_{j,k=1}f_{m,j,k}}_{C^\ell(G)}\leq\varepsilon.
\end{equation}
\end{theorem}

For every $m\in\mathbb N$, put
\begin{equation}\label{e-gue180731}
C^\infty_m(G)=\set{\sum^{d_m}_{j,k=1}c_{j,k}\mathcal{R}_{m,j,k}(g)\in C^\infty(G);\, c_{j,k}\in\Complex, j, k=1,\ldots,d_m}.
\end{equation}
Let $f\in C^\infty(G)$ be a smooth function. We say that $f$ is a smooth $G$-finite function on $G$ if $f=\sum^N_{j=1}f_j$, $N\in\mathbb N$, where $f_j\in C^\infty_{m_j}(G)$, $m_j\in\mathbb N$, $j=1,\ldots,N$. Let $C^\infty_G(G)$ denote the set of all smooth $G$-finite functions on $G$.

\subsection{Fourier analysis on $X$}\label{s-gue180712}

Fix an irreducible unitary representation $\mathcal{R}_m$, for every $g\in G$, put 
\[\chi_m(g):={\rm Tr\,}\left(\mathcal{R}_{m,j,k}(g)\right)^{d_m}_{j,k=1}=\sum^{d_m}_{j=1}\mathcal{R}_{m,j,j}(g).\]
Fix $r=0,1,\ldots,2n+d+1$. We need the following.

\begin{definition}\label{d-gue180712cm}
Let $u\in C^\infty(X,\Lambda^r(\Complex T^*X))$ be a smooth section. The $m$-th Fourier component of $u$ is given by 
\[u_m(x):=d_m\int_G (g^*u)(x)\ol{\chi_m(g)}d\mu(g)\in C^\infty(X,\Lambda^r(\Complex T^*X)).\]
\end{definition}
Note that if $u\in\Omega^{0,q}(X)$, then $u_m\in\Omega^{0,q}(X)$, for every $m=1,2,\ldots$. We have

\begin{theorem}\label{l-gue180713}
Let $u\in C^\infty(X,\Lambda^r(\Complex T^*X))$ be a smooth section. Then,
\begin{equation}\label{e-gue180713mp}
\lim_{N\To+\infty}\sum^N_{m=1}u_m(x)=u(x),\ \ \mbox{for every point $x\in X$},
\end{equation}
\begin{equation}\label{e-gue180715}
(\,u_m\,|\,u_\ell\,)=0\ \ \mbox{if $m\neq\ell$},
\end{equation}
and
\begin{equation}\label{e-gue180715I}
\sum^N_{m=1}\norm{u_m}^2\leq\norm{u}^2,\ \ \forall N\in\mathbb N.
\end{equation}
\end{theorem}

\begin{proof}
Fix $x\in X$. Consider the matrix-valued smooth function on $G$: 
\[f: g\in G\To (g^*u)(x).\]
Let $f_{m,j,k}$ be as in \eqref{e-gue180710mp}. We have 
\begin{equation}\label{e-gue180713mpI}
f_{m,j,k}(g)=d_m\mathcal{R}_{m,j,k}(g)\int (g^*u)(x)\ol{\mathcal{R}_{m,j,k}(g)}d\mu(g).
\end{equation}
By Theorem~\ref{t-gue180710mp}, for every $\varepsilon>0$, there is a $N_0\in\mathbb N$ such that for every $N\geq N_0$, we have 
\begin{equation}\label{e-gue180713mpIIz}
\abs{(g^*u)(x)-\sum^N_{m=1}\sum^{d_m}_{j,k=1}f_{m,j,k}(g)}\leq\varepsilon,\ \ \forall g\in G. 
\end{equation}
Take $g=e_0$ in \eqref{e-gue180713mpIIz}, where $e_0$ denotes the identity element in $G$, we get 
\begin{equation}\label{e-gue180713mpII}
\abs{u(x)-\sum^N_{m=1}\sum^{d_m}_{j,k=1}f_{m,j,k}(e_0)}\leq\varepsilon,
\end{equation}
for every $N\geq N_0$. From \eqref{e-gue180713mpI}, it is easy to see that 
\[\sum^N_{m=1}\sum^{d_m}_{j,k=1}f_{m,j,k}(e_0)=\sum^N_{m=1}u_m(x).\]
From this observation and \eqref{e-gue180713mpII}, we get \eqref{e-gue180713mp}. 

By Theorem~\ref{t-gue180711}, we can check that
\begin{equation}\label{e-gue180715m}
\begin{split}
&\sum^{+\infty}_{m=1}\sum^{d_m}_{j,k=1}\int_G\abs{f_{m,j,k}(g)}^2d\mu(g)\\
&=\sum^{+\infty}_{m=1}\sum^{d_m}_{j,k=1}d_m\abs{\int_G(g^*u)(x)\ol{\mathcal{R}_{m,j,k}(g)}d\mu(g)}^2\\
&=\int_G\abs{(g^*u)(x)}^2d\mu(g),\ \ \mbox{for every $x\in X$}. 
\end{split}
\end{equation}

We notice that for every $p, q\in C^\infty(X,\Lambda^r(\Complex T^*X))$, we have 
\[(\,p\,|\,q\,)=(\,h^*p\,|\,h^*q\,),\ \ \mbox{for every $h\in G$}.\]
Hence, for every $m\in\mathbb N$, $\ell\in\mathbb N$, we have
\begin{equation}\label{e-gue180715mI}
(\,u_m\,|\,u_\ell\,)=\int_G(\,h^*u_m\,|\,h^*u_\ell)d\mu(h).
\end{equation}
Now, for every $h\in G$, we have 
\begin{equation}\label{e-gue180715mII}
\begin{split}
h^*u_m&=d_m\int_G (h^*g^*u)(x)\ol{\chi_m(g)}d\mu(g)\\
&=d_m\int_G((g\circ h)^*u)(x)\ol{\chi_m(g)}d\mu(g)\\
&=d_m\int_G(g^*u)(x)\ol{\chi_m(g\circ h^{-1})}d\mu(g).
\end{split}
\end{equation}
We can check 
\begin{equation}\label{e-gue180715mIII}
\ol{\chi_m(g\circ h^{-1})}=\sum^{d_m}_{j=1}\sum^{d_m}_{s=1}\ol{\mathcal{R}_{m,j,s}(g)}\mathcal{R}_{m,j,s}(h).
\end{equation}
From \eqref{e-gue180715mIII} and \eqref{e-gue180715mII}, we get 
\begin{equation}\label{e-gue180715mIIq}
\begin{split}
h^*u_m=d_m\int_G(g^*u)(x)\Bigr(\sum^{d_m}_{j=1}\sum^{d_m}_{s=1}\ol{\mathcal{R}_{m,j,s}(g)}\mathcal{R}_{m,j,s}(h)d\mu(g)\Bigr)
\end{split}
\end{equation}
and similarly, 
\begin{equation}\label{e-gue180715mIIqa}
\begin{split}
h^*u_\ell=d_\ell\int_G(g^*u)(x)\Bigr(\sum^{d_m}_{j=1}\sum^{d_m}_{s=1}\ol{\mathcal{R}_{\ell,j,s}(g)}\mathcal{R}_{\ell,j,s}(h)d\mu(g)\Bigr).
\end{split}
\end{equation}
From Theorem~\ref{t-gue180711}, we see that 
\begin{equation}\label{e-gue180715pa}
\int \mathcal{R}_{m,j,s}(h)\ol{\mathcal{R}_{\ell,j_1,s_s}(h)}d\mu(h)=0,\ \ \mbox{if $m\neq\ell$, for every $j,s=1,\ldots,d_m$, $j_1, s_1=1,\ldots,d_\ell$}. 
\end{equation}
From \eqref{e-gue180715mIIq}, \eqref{e-gue180715mIIqa} and \eqref{e-gue180715pa}, for every $x\in X$, we obtain
\begin{equation}\label{e-gue180715y}
\int_G\langle\,h^*u_m(x)\,|\,h^*u_\ell(x)\,\rangle d\mu(h)=0\ \ \mbox{if $m\neq\ell$}.
\end{equation}
From \eqref{e-gue180715y} and \eqref{e-gue180715mI}, we get \eqref{e-gue180715}. 

Now, form \eqref{e-gue180715mIIq} and Theorem~\ref{t-gue180711}, for every $x\in X$, we have
\begin{equation}\label{e-gue180715mpf}
\int_G\langle\,h^*u_m(x)\,|\,h^*u_m(x)\,\rangle d\mu(h)=\sum^{d_m}_{j,s=1}d_m\abs{\int_G(g^*u)(x)\ol{\mathcal{R}_{m,j,s}(g)}d\mu(g)}^2.
\end{equation}
From \eqref{e-gue180715mpf} and \eqref{e-gue180715m}, we deduce that 
\begin{equation}\label{e-gue180715mpfI}
\sum^N_{m=1}\int_G\langle\,h^*u_m(x)\,|\,h^*u_m(x)\,\rangle d\mu(h)\leq\int_G\abs{(g^*u)(x)}^2d\mu(g),
\end{equation}
for every $x\in X$ and every $N\in\mathbb N$. From \eqref{e-gue180715mpfI} and \eqref{e-gue180715mI}, we have 
\[\begin{split}
\sum^N_{m=1}\norm{u_m}^2&=\int_X\sum^N_{m=1}\Bigr(\int_G\langle\,h^*u_m(x)\,|\,h^*u_m(x)\,\rangle d\mu(h)dv_X(x)\Bigr)\\
&\leq\int_X\int_G\abs{(g^*u)(x)}^2d\mu(g)dv_X(x)=\norm{u}^2,
\end{split}\]
for every $N\in\mathbb N$. We get \eqref{e-gue180715I}. 
\end{proof}

We pause and introduce some notations. For $s\in\mathbb N$, let $\norm{\cdot}_s$ denote the standard Sobolev norm on $C^\infty(X,\Lambda^r(\Complex T^*X))$ of order $s$ and let $\norm{\cdot}_{C^s(X,\Lambda^r(\Complex T^*X))}$ denote the standard $C^s(X,\Lambda^r(\Complex T^*X))$ norm on $C^\infty(X,\Lambda^r(\Complex T^*X))$. Let $v_j\in C^\infty(X,\Lambda^r(\Complex T^*X))$, $j=1,2,\ldots$. We say that $v_j\To v$ in $C^\infty(X,\Lambda^r(\Complex T^*X))$ topology, where $v\in C^\infty(X,\Lambda^r(\Complex T^*X))$, if for every $s\in\mathbb N$ and every $\varepsilon>0$, there is a $j_0\in\mathbb N$ such that for every $j\geq j_0$, we have 
\[\norm{v_j-v}_{C^s(X,\Lambda^r(\Complex T^*X))}\leq\varepsilon,\ \ \forall j\geq j_0.\]

We can now prove 

\begin{theorem}\label{t-gue180715u}
Let $u\in C^\infty(X,\Lambda^r(\Complex T^*X))$ be a smooth section. Then,
\begin{equation}\label{e-gue180715u}
\mbox{$\lim_{N\To+\infty}\sum^N_{m=1}u_m=u$ in $C^\infty(X,\Lambda^r(\Complex T^*X))$ topology}. 
\end{equation}
\end{theorem}

\begin{proof}
Let $\triangle^{(r)}: C^\infty(X, \Lambda^r(\Complex T^*X))\To C^\infty(X, \Lambda^r(\Complex T^*X))$ be as in \eqref{e-gue180716}. From \eqref{e-gue180711pII}, it is not difficult to see that $\triangle^{(r)}u_m=(\triangle^{(r)}u)_m$ and hence
\begin{equation}\label{e-gue180716I}
(\triangle^{(r)})^ju_m=((\triangle^{(r)})^ju)_m,
\end{equation}
where $((\triangle^{(r)})^ju)_m$ denotes the $m$-th Fourier component of $(\triangle^{(r)})^ju$. For every $N\in\mathbb N$, let 
\[v_N:=\sum^N_{m=1}u_m\in C^\infty(X,\Lambda^r(\Complex T^*X)).\]
From \eqref{e-gue180715}, \eqref{e-gue180715I} and \eqref{e-gue180716I}, we have 
\begin{equation}\label{e-gue180716II}
\begin{split}
&\norm{(\triangle^{(r)})^jv_N}^2=\norm{\sum^N_{m=1}(\triangle^{(r)})^ju_m}^2=\norm{\sum^N_{m=1}((\triangle^{(r)})^ju)_m}^2\\
&=\sum^N_{m=1}\norm{((\triangle^{(r)})^ju)_m}^2\leq\norm{(\triangle^{(r)})^ju}^2, 
\end{split}
\end{equation}
for every $j\in\mathbb N_0$ and every $N\in\mathbb N$. Hence, for every $j\in\mathbb N_0$, we have 
\begin{equation}\label{e-gue180716III}
\lim_{N, M\To+\infty}\norm{((\triangle^{(r)})^j)(v_N-v_M)}^2=0.
\end{equation}
Since $\triangle^{(r)}$ is elliptic, for every $s\in\mathbb N$, there is a constant $C_s>0$ such that 
\begin{equation}\label{e-gue180716a}
\norm{f}^2_s\leq\Bigr(\norm{(\triangle^{(r)})^sf}^2+\norm{f}^2\Bigr),\ \ \mbox{for every $f\in C^\infty(X,\Lambda^r(\Complex T^*X))$}.
\end{equation}
From \eqref{e-gue180716a} and \eqref{e-gue180716III}, we conclude that for every $s\in\mathbb N$, there is a $\Td u_s\in H^s(X,\Lambda^r(\Complex T^*X))$ such that 
$v_N\To\Td u_s$ in $H^s(X,\Lambda^r(\Complex T^*X))$. From the Sobolev embedding theorem, there is a $s_0\in\mathbb N$ such that for all $s\geq s_0$, $\Td u_s\in C^0(X, \Lambda^r(\Complex T^*X))$ and $v_N\To\Td u_s$ in $C^0(X,\Lambda^r(\Complex T^*X))$. From \eqref{e-gue180713mp}, we see that $v_N\To u$ pointwise and hence 
\[u=\Td u_s\ \ \mbox{for all $s\in\mathbb N$ with $s\geq s_0$}.\]
We have proved that $v_N\To u$ in $H^s(X,\Lambda^r(\Complex T^*X))$, for all $s\in\mathbb N$ with $s\geq s_0$. By Sobolevs embedding theorem, $v_N\To u$ in $C^\infty(X,\Lambda^r(\Complex T^*X))$ topology. The theorem follows. 
\end{proof}

For every $m\in\mathbb N$, put 
\begin{equation}\label{e-gue180716z}
\Omega^{0,q}_m(X):=\set{f\in\Omega^{0,q}(X);\, \mbox{there is a $F\in\Omega^{0,q}(X)$ such that $f=F_m$ on $X$}}. 
\end{equation}
Let $L^2_{(0,q),m}(X)$ be the $L^2$ completion of $\Omega^{0,q}_m(X)$ with resect to $(\,\cdot\,|\,\cdot\,)$. We denote $C^\infty_m(X):=\Omega^{0,0}_m(X)$, $L^2_m(X):=L^2_{(0,0),m}(X)$. We need 

\begin{lemma}\label{l-gue180716z}
Fix $m\in\mathbb N$, $\ell\in\mathbb N$ with $m\neq\ell$. Let $\mathcal{R}_{m,j,k}$ be any matrix element of $\mathcal{R}_m$ and let $\mathcal{R}_{\ell,j_1,k_2}$ be any matrix element of $\mathcal{R}_\ell$. Let $A\in\Omega^{0,q}(X)$. Then,
\begin{equation}\label{e-gue180716am}
\int_G(h^*g^*A)(x)\ol{\mathcal{R}_{m,j,k}(g)}\,\ol{\mathcal{R}_{\ell,j_1,k_1}(h)}d\mu(g)d\mu(h)=0,\ \ \mbox{for every $x\in X$}. 
\end{equation}
\end{lemma}

\begin{proof}
We have
\begin{equation}\label{e-gue180716amI}
\begin{split}
&\int_G(h^*g^*A)(x)\ol{\mathcal{R}_{m,j,k}(g)}\,\ol{\mathcal{R}_{\ell,j_1,k_1}(h)}d\mu(g)d\mu(h)\\
&=\int_G((g\circ h)^*A)(x)\ol{\mathcal{R}_{m,j,k}(g)}\,\ol{\mathcal{R}_{\ell,j_1,k_1}(h)}d\mu(g)d\mu(h)\\
&=\int_G(g^*A)(x)\ol{\mathcal{R}_{m,j,k}(g\circ h^{-1})}\,\ol{\mathcal{R}_{\ell,j_1,k_1}(h)}d\mu(g)d\mu(h)\\
&=\int_G(g^*A)(x)\Bigr(\sum^{d_m}_{s=1}\ol{\mathcal{R}_{m,j,s}(g)}\mathcal{R}_{m,k,s}(h)\Bigr)\ol{\mathcal{R}_{\ell,j_1,k_1}(h)}d\mu(g)d\mu(h). 
\end{split}
\end{equation}
By Theorem~\ref{t-gue180711}, we have 
\begin{equation}\label{e-gue180716amII}
\int_G\mathcal{R}_{m,k,s}(h)\ol{\mathcal{R}_{\ell,j_1,k_1}(h)}d\mu(h)=0,\ \ \mbox{for every $s=1,\ldots,d_m$}.
\end{equation}
From \eqref{e-gue180716amII} and \eqref{e-gue180716amI}, the lemma follows. 
\end{proof}

We can prove 

\begin{theorem}\label{t-gue180716ap}
Fix $m\in\mathbb N$ and let $\mathcal{R}_{m,j,k}$ be any matrix element of $\mathcal{R}_m$. Let $A\in\Omega^{0,q}(X)$ and set 
$f(x)=\int_G(g^*A)(x)\ol{\mathcal{R}_{m,j,k}(g)}d\mu(g)$. Then, $f=f_m$ on $X$. 
\end{theorem}

\begin{proof}
From Theorem~\ref{t-gue180715u}, we have 
\begin{equation}\label{e-gue180716ap}
\lim_{N\To+\infty}\sum^{N}_{\ell=1}f_\ell(x)=f(x)\ \ \mbox{in $C^\infty(X,\Lambda^q(\Complex T^*X))$ topology}.
\end{equation}
Now, for every $\ell\neq m$, from \eqref{e-gue180716am}, we have 
\begin{equation}\label{e-gue180716apI}
f_\ell(x)=\int_G(h^*g^*A)(x)\ol{\mathcal{R}_{m,j,k}(g)}\,\ol{\chi_{\ell}(h)}d\mu(g)d\mu(h)=0.
\end{equation}
From \eqref{e-gue180716apI} and \eqref{e-gue180716ap}, we get $f=f_m$ and the theorem follows. 
\end{proof}

\begin{corollary}\label{c-gue180716ap}
Let $u\in\Omega^{0,q}(X)$ be a \((0,q)\)-form. Then, $u\in\Omega^{0,q}_m(X)$ if and only if $u=u_m$ on $X$. 
\end{corollary}

Let $T_j\in C^\infty(X,TX)$, $j=1,\ldots,d$, be as in \eqref{e-gue180713}. Let $u\in\Omega^{0,q}(X)$. For each $j=1,\ldots,d$, we denote $T_ju:=L_{T_j}u\in\Omega^{0,q}(X)$, where $L_{T_j}$ denotes the Lie derivative of $u$ along $T_j$. We have 

\begin{theorem}\label{t-gue180716apI}
Fix $j=1,\ldots,d$, and $m\in\mathbb N$. We have $T_ju\in\Omega^{0,q}_m(X)$, for every $u\in\Omega^{0,q}_m(X)$ and there is a constant $C>0$ such that 
\begin{equation}\label{e-gue180716bm}
\norm{T_ju}\leq C\norm{u},\ \ \mbox{for every $u\in\Omega^{0,q}_m(X)$}. 
\end{equation}
\end{theorem}

\begin{proof}
Let $u\in\Omega^{0,q}_m(X)$ be a \((0,q)\)-form. We have 
\begin{equation}\label{e-gue180716q}
\begin{split}
(T_ju)(x)&=\frac{\pr}{\pr t}\Bigr(((e^{t\xi_j})^*u)(x)\Bigr)|_{t=0}\\
&=\frac{\pr}{\pr t}\Bigr(\int_G((e^{t\xi_j})^*g^*)u)(x)\ol{\chi_m(g)}d\mu(g)\Bigr)|_{t=0}\\
&=\frac{\pr}{\pr t}\Bigr(\int_G((ge^{t\xi_j})^*u)(x)\ol{\chi_m(g)}d\mu(g)\Bigr)|_{t=0}\\
&=\frac{\pr}{\pr t}\Bigr(\int_G(g^*u)(x)\ol{\chi_m(g\circ(e^{t\xi_j})^{-1})}d\mu(g)\Bigr)|_{t=0}\\
&=\frac{\pr}{\pr t}\Bigr(\int_G(g^*u)(x)\sum^{d_m}_{k=1}\sum^{d_m}_{s=1}\ol{\mathcal{R}_{m,k,s}(g)}\mathcal{R}_{m,k,s}(e^{t\xi_j})d\mu(g)\Bigr)|_{t=0}\\
&=\int_G(g^*u)(x)\sum^{d_m}_{k=1}\sum^{d_m}_{s=1}c_{k,s}\ol{\mathcal{R}_{m,k,s}(g)}d\mu(g),
\end{split}
\end{equation}
where $c_{k,s}=\frac{\pr}{\pr t}\Bigr(\mathcal{R}_{m,k,s}(e^{t\xi_j})\Bigr)|_{t=0}\in\Complex$, $k,s=1,\ldots,d_m$. From \eqref{e-gue180716q} and Theroem~\ref{t-gue180716ap}, we deduce that $T_ju\in\Omega^{0,q}_m(X)$. 

Now, from \eqref{e-gue180716q}, we have 
\begin{equation}\label{e-gue189716qI}
\begin{split}
\norm{T_ju}^2&=\int_X\abs{\int_G(g^*u)(x)\sum^{d_m}_{k=1}\sum^{d_m}_{s=1}c_{k,s}\ol{\mathcal{R}_{m,k,s}(g)}d\mu(g)}^2dv_X(x)\\
&\leq\int_X\Bigr(\int_G\abs{(g^*u)(x)}^2d\mu(g)\int_G\abs{\sum^{d_m}_{k=1}\sum^{d_m}_{s=1}c_{k,s}\ol{\mathcal{R}_{m,k,s}(g)}}^2d\mu(g)\Bigr)dv_X(x)\\
&\leq C\int_X\int_G\abs{(g^*u)(x)}^2d\mu(g)dv_X(x)=C\norm{u}^2,
\end{split}
\end{equation}
where $C=\int_G\abs{\sum^{d_m}_{k=1}\sum^{d_m}_{s=1}c_{k,s}\ol{\mathcal{R}_{m,k,s}(g)}}^2d\mu(g)$. We get \eqref{e-gue180716bm} and the theorem follows. 
\end{proof}

We have 

\begin{theorem}\label{t-gue180720}
Fix $j=1,\ldots,d$, and $m\in\mathbb N$. Let $s\in\mathbb N$. Then, there is a constant $C_s>0$ such that 
\begin{equation}\label{e-gue180720m}
\norm{T_ju}_s\leq C_s\norm{u}_s,\ \ \mbox{for every $u\in\Omega^{0,q}_m(X)$},
\end{equation}
where $\norm{\cdot}_s$ denotes the standard Sobolev norm of order $s$. 
\end{theorem}

\begin{proof}
Fix $s\in\mathbb N$. 
Since $\sqrt{\triangle^{(r)}+I}$ is a classical elliptic pseudodifferential operator of order one (see \eqref{e-gue180719I}), there is a constant $\hat C_s>1$ such that 
\begin{equation}\label{e-gue180720}
\frac{1}{\hat C_s}\norm{(\sqrt{\triangle^{(r)}+I})^su}\leq\norm{u}_s\leq\hat C_s\norm{(\sqrt{\triangle^{(r)}+I})^su},
\end{equation}
for every $u\in C^\infty(X,\Lambda^q(\Complex T^*X))$. From \eqref{e-gue180719II}, we have 
\begin{equation}\label{e-gue180720I}
(\sqrt{\triangle^{(r)}+I})^s(T_ju)=T_j((\sqrt{\triangle^{(r)}+I})^su),\ \ \forall u\in\Omega^{0,q}(X).
\end{equation}
From \eqref{e-gue180720}, \eqref{e-gue180720I}, \eqref{e-gue180719II} and \eqref{e-gue180716bm}, we have 
\begin{equation}\label{e-gue180720II}
\begin{split}
\norm{T_ju}_s&\leq\hat C_s\norm{(\sqrt{\triangle^{(r)}+I})^s(T_ju)}=\hat C_s\norm{T_j((\sqrt{\triangle^{(r)}+I})^su)}\\
&\leq\hat C_sC\norm{(\sqrt{\triangle^{(r)}+I})^su}\leq(\hat C_s)^2C \norm{u}_s,
\end{split}
\end{equation}
for every $u\in\Omega^{0,q}_m(X)$, where $C>0$ and $\hat C_s>0$ are constants as in \eqref{e-gue180716bm} and \eqref{e-gue180720} respectively. The theorem follows. 
\end{proof}

We introduce some definitions. 

\begin{definition}\label{d-gue180731}
Let $f\in C^\infty(X)$ be a smooth function. We say that $f$ is a $G$-finite smooth function on $X$ if there is a $\Td f\in C^\infty(X)$ and $h(g)\in C^\infty_G(G)$ such that 
\[f(x)=\int\Td f(g\circ x)h(g)d\mu(g).\]
Let $C^\infty_G(X)$ be the set of all $G$-finite smooth functions on $X$.
\end{definition}

For the meaning of $C^\infty_G(G)$, we refer the reader to the discussion after \eqref{e-gue180731}.  Let $f\in C^\infty_G(X)$ be a smooth \(G\)-finite function. From Theorem~\ref{t-gue180716ap}, we see that $f\in C^\infty_G(X)$ if and only if $f=\sum^N_{j=1}f_j$, $N\in\mathbb N$, where $f_j\in C^\infty_{m_j}(X)$, for some $m_j\in\mathbb N$, $j=1,\ldots,N$. 

\subsection{Fourier components of the Kohn Laplacian}\label{s-gue180716}

We fix $m\in\mathbb N$. From \eqref{e-gue180711p}, we see that 
\[\ddbar_b: \Omega^{0,q}_m(X)\To\Omega^{0,q+1}_m(X).\]
We extend $\ddbar_b$ to $L^2_{(0,q),m}(X)$ by
\[\begin{split}
&\ddbar_b:{\rm Dom\,}\ddbar_b\subset L^2_{(0,q),m}(X)\To L^2_{(0,q+1),m}(X),\\
&{\rm Dom\,}\ddbar_b:=\set{u\in  L^2_{(0,q),m}(X);\, \ddbar_bu\in L^2_{(0,q+1),m}(X)}.
\end{split}\]
Let 
\[\ol{\pr}^*_b:{\rm Dom\,}\ol{\pr}^*_b\subset L^2_{(0,q+1),m}(X)\To L^2_{(0,q),m}(X)\]
be the $L^2$ adjoint of $\ddbar_b$ with respect to $(\,\cdot\,|\,\cdot\,)$. The $m$-th Fourier component of the Kohn Laplacian is given by 
\begin{equation}\label{e-gue180716w}
\begin{split}
&\Box^{(q)}_{b,m}=\ddbar_b\,\ol{\pr}^*_b+\ol{\pr}^*_b\,\ddbar_b: {\rm Dom\,}\Box^{(q)}_{b,m}\subset L^2_{(0,q),m}(X)\To L^2_{(0,q),m}(X),\\
&{\rm Dom\,}\Box^{(q)}_{b,m}=\set{u\in L^2_{(0,q),m}(X);\, u\in {\rm Dom\,}\ddbar_b\bigcap{\rm Dom\,}\ol{\pr}^*_b, \ddbar_bu\in{\rm Dom\,}\ol{\pr}^*_b, \ol{\pr}^*_bu\in{\rm Dom\,}\ddbar_b}.
\end{split}
\end{equation}

For every $j=1,\ldots,d$, let $T^*_j: \Omega^{0,q}(X)\To\Omega^{0,q}(X)$ be the formal adjoint of $T_j$ with respect to $(\,\cdot\,|\,\cdot\,)$. Put 
\begin{equation}\label{e-gue180720mp}
\hat\Box^{(q)}_b:=\Box^{(q)}_b+T^*_1T_1+\cdots+T^*_dT_d: \Omega^{0,q}(X)\To\Omega^{0,q}(X).
\end{equation}
Since the Levi from $\mathcal{L}_{\omega_0,x}$ is positive definite at every point $x\in X$ and $n\geq 2$, we can repeat the proof of Theorem 8.4.2 in~\cite{CS01} and deduce the following 

\begin{theorem}\label{t-gue180720mp}
Let $q\geq 1$ be a positive integer. For every $s\in\mathbb N_0$, there is a constant $ C_s>0$ such that 
\begin{equation}\label{e-gue180720mpI}
\norm{u}_{s+1}\leq C_s\Bigr(\norm{\hat\Box^{(q)}_bu}_s+\norm{u}_s\Bigr), 
\end{equation}
for every $u\in\Omega^{0,q}(X)$. 
\end{theorem}

From now on, we fix $m\in\mathbb N$. We can prove 

\begin{theorem}\label{t-gue180720mpI}
Let $q\geq 1$ be a positive integer. For every $s\in\mathbb N_0$, there is a constant $C_s>0$ such that 
\begin{equation}\label{e-gue180720mpz}
\norm{u}_{s+1}\leq C_s\Bigr(\norm{\Box^{(q)}_{b,m}u}_s+\norm{u}_s\Bigr), 
\end{equation}
for every $u\in\Omega^{0,q}_m(X)$. 
\end{theorem}

\begin{proof}
Fix $s\in\mathbb N_0$ and let $u\in\Omega^{0,q}_m(X)$. We have 
\begin{equation}\label{e-gue180720mpII}
\norm{\hat\Box^{(q)}_bu}_s\leq C\Bigr(\norm{\Box^{(q)}_{b,m}u}_s+\sum^d_{j=1}\norm{T^*_jT_ju}_s\Bigr),
\end{equation}
where $C>0$ is a constant independent of $u$. From \eqref{e-gue180720m}, we see that 
\begin{equation}\label{e-gue180720mpIII}
\norm{T^*_jT_ju}_s\leq\hat C\norm{u}_s,\ \ j=1,\ldots,d,
\end{equation}
where $\hat C>0$ is a constant independent of $u$. From \eqref{e-gue180720mpII}, \eqref{e-gue180720mpIII} and \eqref{e-gue180720mpI}, we get \eqref{e-gue180720mpz}. 
\end{proof}

For every $s\in\mathbb N$, let $H^s_m(X, T^{*0,q}X)$ be the completion of $\Omega^{0,q}_m(X)$ with respect to $\norm{\cdot}_s$. From Theorem~\ref{t-gue180720mpI}, 
we can repeat the technique of elliptic regularization (see the proof of Theroem 8.4.2 in~\cite{CS01}) and conclude that 

\begin{theorem}\label{t-gue180720mpII}
Let $q\geq 1$ be a positive integer and $u\in{\rm Dom\,}\Box^{(q)}_{b,m}$ with $\Box^{(q)}_{b,m}u=v\in L^2_{(0,q),m}(X)$. If $v\in H^s_m(X, T^{*0,q}X)$, $s\in\mathbb N_0$, then $u\in H^{s+1}_m(X,T^{*0,q}X)$ and there is a constant $C_s>0$ independent of $u$, $v$, such that 
\[\norm{u}_{s+1}\leq C_s\Bigr(\norm{\Box^{(q)}_{b,m}u}_s+\norm{u}_s\Bigr).\]
\end{theorem}

From Theorem~\ref{t-gue180720mpII} and some standard argument in functional analysis, we get 

\begin{theorem}\label{t-gue180720mpIII}
For $q\geq 1$, we have that
\[\Box^{(q)}_{b,m}=\ddbar_b\,\ol{\pr}^*_b+\ol{\pr}^*_b\,\ddbar_b: {\rm Dom\,}\Box^{(q)}_{b,m}\subset L^2_{(0,q),m}(X)\To L^2_{(0,q),m}(X)\]
has closed range.
\end{theorem}
 
 For $q\geq1$, let $N^{(q)}_m: L^2_{(0,q),m}(X)\To {\rm Dom\,}\Box^{(q)}_{b,m}$ be the partial inverse of $\Box^{(q)}_{b,m}$ and let 
 $S^{(q)}_m: L^2_{(0,q),m}(X)\To{\rm Ker\,}\Box^{(q)}_{b,m}$ be the Szeg\H{o} projection, i.e., the orthogonal projection onto ${\rm Ker\,}\Box^{(q)}_{b,m}$ with respect to $(\cdot\,|\,\cdot\,)$. We have 
 \begin{equation}\label{e-gue180720p}
 \begin{split}
 &\Box^{(q)}_{b,m}N^{(q)}_m+S^{(q)}_m=I\ \ \mbox{on $L^2_{(0,q),m}(X)$},\\
 &N^{(q)}_m\Box^{(q)}_{b,m}+S^{(q)}_m=I\ \ \mbox{on ${\rm Dom\,}\Box^{(q)}_{b,m}$}.
 \end{split}
 \end{equation}
From Theorem~\ref{t-gue180720mpII}, we obtain 

\begin{theorem}\label{t-gue180720p}
For $q\geq 1$, we have that $N^{(q)}_m: H^s_m(X,T^{*0,q}X)\To H^{s+1}_m(X,T^{*0,q}X)$ is continuous, for every $s\in\mathbb N_0$. 
\end{theorem}

We consider the case $q=0$. Let 
\begin{equation}\label{e-gue180731w}
S_m:=S^{(0)}_m: L^2_{m}(X)\To{\rm Ker\,}\Box^{(0)}_{b,m}
\end{equation}
 be the Szeg\H{o} projection. We need 

\begin{theorem}\label{t-gue180720q}
We have 
\begin{equation}\label{e-gue180720r}
S_m=I-\ol{\pr}^*_bN^{(1)}_m\ddbar_b\ \ \mbox{on $C^\infty_m(X)$}. 
\end{equation}
\end{theorem}

\begin{proof}
Put 
\[\hat S_m:=I-\ol{\pr}^*_bN^{(1)}_m\ddbar_b: C^\infty_m(X)\To C^\infty_m(X).\]
Take $u\in C^\infty_m(X)$. From \eqref{e-gue180720p}, we have 
\begin{equation}\label{e-gue180720rI}
\begin{split}
&\ddbar_b\hat S_mu\\
&=\ddbar_bu-\ddbar_b\,\ol{\pr}^*_bN^{(1)}_m\ddbar_bu\\
&=\ddbar_bu-\Bigr(N^{(1)}_m\Box^{(1)}_{b,m}+S^{(1)}_m\Bigr)\ddbar_b\,\ol{\pr}^*_bN^{(1)}_m\ddbar_bu.
\end{split}
\end{equation}
Since $S^{(1)}_m\ddbar_b=0$ and $\Box^{(1)}_{b,m}\ddbar_b\,\ol{\pr}^*_b=\ddbar_b\,\ol{\pr}^*_b\,\ddbar_b\,\ol{\pr}^*_b=\ddbar_b\,\ol{\pr}^*_b\Box^{(1)}_{b,m}$, we have 
\begin{equation}\label{e-gue180720rII}
\begin{split}
&\Bigr(N^{(1)}_m\Box^{(1)}_{b,m}+S^{(1)}\Bigr)\ddbar_b\,\ol{\pr}^*_bN^{(1)}_m\ddbar_b\\
&=N^{(1)}_m\Box^{(1)}_{b,m}\ddbar_b\,\ol{\pr}^*_bN^{(1)}_m\ddbar_b\\
&=N^{(1)}_m\ddbar_b\,\ol{\pr}^*_b\,\ddbar_b\,\ol{\pr}^*_bN^{(1)}_m\ddbar_b\\
&=N^{(1)}_m\ddbar_b\,\ol{\pr}^*_b\Box^{(1)}_{b,m}N^{(1)}_m\ddbar_b\\
&=N^{(1)}_m\ddbar_b\,\ol{\pr}^*_b\,\ddbar_b\ \ \mbox{(here we used \eqref{e-gue180720p})}\\
&=N^{(1)}_m\Box^{(1)}_{b,m}\ddbar_b\\
&=\ddbar_b \mbox{(here we used \eqref{e-gue180720p})}.
\end{split}
\end{equation}
From \eqref{e-gue180720rI} and \eqref{e-gue180720rII}, we get 
\begin{equation}\label{e-gue180720rIII}
\ddbar_b\hat S_mu=0.
\end{equation} 
For any $v\in{\rm Ker\,}\ddbar_b$, we have 
\begin{equation}\label{e-gue180720y}
(\,v\,|\,\ol{\pr}^*_bN^{(1)}_m\ddbar_bu\,)=(\,\ddbar_bv\,|\,N^{(1)}_m\ddbar_bu\,)=0.
\end{equation}
From \eqref{e-gue180720rIII} and \eqref{e-gue180720y}, we deduce that
\[u=\hat S_m+(I-\hat S_m)u=(I-\ol{\pr}^*_bN^{(1)}_m\ddbar_b)u+\ol{\pr}^*_bN^{(1)}_m\ddbar_bu\]
is the orthogonal decomposition and hence $\hat S_m=S_m$. 
\end{proof}

We can now prove 

\begin{theorem}\label{t-gue180720f}
We have 
\begin{equation}\label{e-gue180720f}
S_m: C^\infty_m(X)\To C^\infty_m(X)
\end{equation}
and for every $\ell\in\mathbb N$, there are $K\in\mathbb N$ and $C>0$ such that 
\begin{equation}\label{e-gue180720fI}
\norm{(I-S_m)u}_{C^\ell(X)}\leq C\norm{\ddbar_bu}_{C^K(X,\Complex T^*X)},\ \ \mbox{for every $u\in C^\infty_m(X)$}. 
\end{equation}
\end{theorem}

\begin{proof}
From \eqref{e-gue180720r} and Theorem~\ref{t-gue180720p}, we get \eqref{e-gue180720f}. 

Let $\ell\in\mathbb N$ be a nonnegative integer. By the Sobolev inequalities, there is a $K\in\mathbb N$ and a constant $C>0$ such that 
\begin{equation}\label{e-gue180720qm}
\norm{(I-S_m)u}_{C^\ell(X)}\leq C\norm{(I-S_m)u}_K,\ \ \mbox{for evert $u\in C^\infty(X)$}. 
\end{equation}
From Theorem~\ref{t-gue180720p}, we see that 
\[\ol{\pr}^*_bN^{(1)}_m: H^s_m(X,T^{*0,1}X)\To H^s_m(X)\]
is continuous, for every $s\in\mathbb N_0$. From this observation and \eqref{e-gue180720r},  we have 
\begin{equation}\label{e-gue180720qmI}
\norm{(I-S_m)u}_K=\norm{(\ol{\pr}^*_bN^{(1)}_m\ddbar_b)u}_K\leq\hat C\norm{\ddbar_bu}_{K}\leq\Td C\norm{\ddbar_bu}_{C^K(X,\Complex T^*X)}, 
\end{equation}
for every $u\in C^\infty_m(X)$, where $\hat C>0$ and $\Td C>0$ are constants. From \eqref{e-gue180720qmI} and \eqref{e-gue180720qm}, we get \eqref{e-gue180720fI}. 
\end{proof}

\section{$G$-equivariant embedding theorems}\label{s-gue180722}

In this section, we will prove Theorem~\ref{t-gue180803m}. The idea is to construct global \(G\)-finite CR functions which separate points and define local embeddings. Since \(X\) is compact, this leads to an embedding of \(X\) by \(G\)-finite CR functions.

 Until further notice, we fix $p\in X$. Let $Z_1\in C^\infty(X,T^{1,0}X),\ldots,Z_n\in C^\infty(X,T^{1,0}X)$ be smooth sections such that for every $x\in X$, $\set{Z_1(x),\ldots,Z_n(x)}$ is an orthonormal basis for $T^{1,0}_xX$ and the Levi form $\mathcal{L}_{\omega_0,p}$ is diagonalized with respect to  $\set{Z_1(p),\ldots,Z_n(p)}$, that is, 
\begin{equation}\label{e-gue180722}
\mathcal{L}_{\omega_0,p}(Z_j(p), \ol Z_\ell(p))=\delta_{j,\ell}\lambda_j,\ \ j, \ell=1,\ldots,n,
\end{equation}
where $\delta_{j,\ell}=1$ if $j=\ell$, $\delta_{j,\ell}=0$ if $j\neq\ell$. 
Note that $\lambda_j>0$, for every $j=1,\ldots,n$. 
Let $e_0$ denote the identity element of $G$. By the Frobenius theorem, there exist local coordinates $\theta=(\theta_1,\ldots,\theta_d)$ of $G$ defined in a neighborhood $V$ of $e_0$ with $v(e_0)=(0,\ldots,0)$ and local coordinates $x=(x_1,\ldots,x_{2n+1+d})$ of $X$ defined in a neighborhood $U$ of $p$ with $x(p)=0$ such that 
\begin{equation}\label{e-gue180722I}
\begin{split}
&(\theta_1,\ldots,\theta_d)\circ (x_1,\ldots,,x_{2n+1},0,\ldots,0)\\
&=(x_{1},\ldots,x_{2n+1},\theta_1,\ldots,\theta_d),\ \ \forall (\theta_1,\ldots,\theta_d)\in V,\ \ \forall(x_1,\ldots,,x_{2n+1},0,\ldots,0)\in U,
\end{split}
\end{equation}
where we identify $V$ as an open set of original point in $\Real^d$, 
\begin{equation}\label{e-gue180722II}
\underline{\mathfrak{g}}={\rm span\,}\set{\frac{\pr}{\pr x_{2n+2}},\ldots,\frac{\pr}{\pr x_{2n+1+d}}},
\end{equation}
and 
\begin{equation}\label{e-gue180722III}
\begin{split}
&Z_j(p)=\frac{\pr}{\pr z_j}(p),\ \ j=1,\ldots,n,\\
&T(p)=\frac{\pr}{\pr x_{2n+1}}(p),
\end{split}
\end{equation}
where $\frac{\pr}{\pr z_j}=\frac{1}{2}(\frac{\pr}{\pr x_{2j-1}}-i\frac{\pr}{\pr x_{2j}})$, $j=1,\ldots,n$. Put $z_j=x_{2j-1}+ix_{2j}$, $j=1,\ldots,d$, $z=(z_1,\ldots,z_n)$. 
We write 
\begin{equation}\label{e-gue180806}
\begin{split}
Z_j&=\frac{\pr}{\pr z_j}+\sum^n_{s=1}(a_{j,s}z_s+b_{j,s}\ol z_s)\frac{\pr}{\pr x_{2n+1}}+\Bigr(c_jx_{2n+1}+\sum^{2n+d+1}_{s=2n+2}c_{j,s}x_s\Bigr)\frac{\pr}{\pr x_{2n+1}}\\
&\quad+\sum_{1\leq j\leq 2n+1+d, j\neq 2n+1}O(\abs{x})\frac{\pr}{\pr x_{j}}+O(\abs{x}^2),\ \ j=1,\ldots,n,
\end{split}
\end{equation}
where $a_{j,s}\in\Complex$, $b_{j,s}\in\Complex$, $j, s=1,\ldots,n$, $c_j\in\Complex$, $c_{j,s}\in\Complex$, $j=1,\ldots,n$, $s=2n+2,\ldots,2n+d+1$. 
From \eqref{e-gue180722}, \eqref{e-gue180722III} and $[Z_j, Z_\ell]\in C^\infty(X, T^{1,0}X)$, for every $j, \ell=1,\ldots,n$, it is straightforward to check that 
\begin{equation}\label{e-gue180806I}
\begin{split}
&a_{j,\ell}=a_{\ell,j},\ \ j, \ell=1,2,\ldots,n,\\
&\ol{b_{\ell,j}}-b_{j,\ell}=-2i\lambda_j\delta_{j,\ell},\ \ j, \ell=1,\ldots,n.
\end{split}
\end{equation}
From \eqref{e-gue180806I}, we can change $x_{2n+1}$ to 
\[x_{2n+1}-\frac{1}{2}\Bigr(\sum^n_{j,\ell=1}a_{j,\ell}z_jz_\ell+\sum^n_{j,\ell=1}\ol{a_{j,\ell}}\,\ol z_j\,\ol z_\ell+
\sum^n_{j,\ell=1}b_{j,\ell}z_j\ol z_\ell+\sum^n_{j,\ell=1}\ol{b_{j,\ell}}\,\ol z_jz_\ell\Bigr)\]
and by some straightforward calculation, we have
\begin{equation}\label{e-gue180722mp}
\begin{split}
Z_j&=\frac{\pr}{\pr z_j}+i\lambda_j\ol z_j\frac{\pr}{\pr x_{2n+1}}+\Bigr(c_jx_{2n+1}+\sum^{2n+d+1}_{s=2n+2}c_{j,s}x_s\Bigr)\frac{\pr}{\pr x_{2n+1}}\\
&\quad+\sum_{1\leq j\leq 2n+1+d, j\neq 2n+1}O(\abs{x})\frac{\pr}{\pr x_{j}}+O(\abs{x}^2),\ \ j=1,\ldots,n.
\end{split}
\end{equation}
From $[Z_j,\frac{\pr}{\pr x_s}]\subset C^\infty(X, T^{1,0}X)$, for every $j=1,\ldots,n$, $s=2n+2,\ldots,2n+d+1$, we deduce that $c_{j,s}=0$, for every $j=1,\ldots,n$, $s=2n+2,\ldots,2n+d+1$. Hence, 
\begin{equation}\label{e-gue180723}
\begin{split}
Z_j&=\frac{\pr}{\pr z_j}+i\lambda_j\ol z_j\frac{\pr}{\pr x_{2n+1}}+c_jx_{2n+1}\frac{\pr}{\pr x_{2n+1}}\\
&\quad+\sum_{1\leq j\leq 2n+1+d, j\neq 2n+1}O(\abs{x})\frac{\pr}{\pr x_{j}}+O(\abs{x}^2),\ \ j=1,\ldots,n,
\end{split}
\end{equation}
and 
\begin{equation}\label{e-gue180724}
\begin{split}
\ol Z_j&=\frac{\pr}{\pr\ol z_j}-i\lambda_jz_j\frac{\pr}{\pr x_{2n+1}}+\ol c_jx_{2n+1}\frac{\pr}{\pr x_{2n+1}}\\
&\quad+\sum_{1\leq j\leq 2n+1+d, j\neq 2n+1}O(\abs{x})\frac{\pr}{\pr x_{j}}+O(\abs{x}^2),\ \ j=1,\ldots,n.
\end{split}
\end{equation}
Until further notice, we work with the local coordinates $x$ defined in an open set $U$ of $p$. We identify $V$ as an open neighborhood of the origin in $\Real^d$ and we assume that $U=\Td U\times V$, where $\Td U$ is an open neighborhood of the origin in $\Real^{2n+1}$. 

We say that $g$ is a homogeneous polynomial of degree $\ell\in\mathbb N_0$ on $U$ if $g$ is the finite sum 
\[g=\sum_{\abs{\alpha}+\abs{\beta}+\abs{\gamma}+\gamma_0=\ell, \alpha, \beta\in\mathbb N^n_0, \gamma_0\in\mathbb N_0, \gamma\in\mathbb N^{d}_0}c_{\alpha,\beta,\gamma_0,\gamma}\ol z^{\alpha} z^{\beta}x^{\gamma_0}_{2n+1}\theta^{\gamma},\]
where $\alpha=(\alpha_1,\ldots,\alpha_n)$, $\beta=(\beta_1,\ldots,\beta_n)$, $\gamma=(\gamma_1,\ldots,\gamma_{d})$, $\abs{\alpha}=\sum^n_{j=1}\alpha_j$, 
$\abs{\beta}=\sum^n_{j=1}\beta_j$, $\abs{\gamma}=\sum^{d}_{j=1}\gamma_j$, 
$\ol z^\alpha:=\ol z_1^{\alpha_1}\cdots \ol z_n^{\alpha_n}$, $z^\beta:=z_1^{\beta_1}\cdots z_n^{\beta_n}$, $\theta^\gamma:=\theta_1^{\gamma_1}\cdots\theta_d^{\gamma_{d}}$, $c_{\alpha,\beta,\gamma,\gamma_0,\gamma}\in\Complex$, for every $\alpha, \beta\in\mathbb N^n_0, \gamma_0\in\mathbb N_0, \gamma\in\mathbb N^{d}_0$. 
Let $P_\ell(U)$ be the set of all homogeneous polynomials of degree $\ell$. 

\begin{theorem}\label{t-gue180723}
We can find $\phi_\ell\in P_\ell(U)$, $\ell=1,2,3,\ldots$, with 
\begin{equation}\label{e-gue180724I}
\begin{split}
&\phi_1=x_{2n+1},\\
&\phi_2=i\Bigr(\sum^n_{j=1}\lambda_j\abs{z_j}^2+\abs{x_{2n+1}}^2+\sum^d_{j=1}\abs{\theta_j}^2\Bigr)-\sum^n_{j=1}(\ol c_j\ol z_jx_{2n+1}+c_jz_jx_{2n+1}),
\end{split}
\end{equation}
where $c_j$, $j=1,\ldots,n$, are as in \eqref{e-gue180724}, such that 
\begin{equation}\label{e-gue180724II}
\ol Z_j\Bigr(\sum^N_{\ell=1}\phi_\ell\Bigr)=O(\abs{x}^N),\ \ \mbox{for every $j=1,\ldots,n$, for every $N\in\mathbb N$}.
\end{equation}
\end{theorem}

\begin{proof}
Let $\phi_1\in P_1(U)$, $\phi_2\in P_2(U)$ be as in \eqref{e-gue180724I}. From \eqref{e-gue180724}, it is easy to see that 
\[\ol Z_j\Bigr(\sum^2_{\ell=1}\phi_\ell\Bigr)=O(\abs{x}^2),\ \ \mbox{for every $j=1,\ldots,n$}.\]
We assume that we can find $\phi_\ell\in P_\ell(U)$, $\ell=1,2,3,\ldots, N$, where $\phi_1$ and $\phi_2$ are as in \eqref{e-gue180724I}, such that 
\begin{equation}\label{e-gue180729m}
\ol Z_j\Bigr(\sum^N_{\ell=1}\phi_\ell\Bigr)=O(\abs{x}^N),\ \ \mbox{for every $j=1,\ldots,n$}.
\end{equation}
We write 
\begin{equation}\label{e-gue180729}
\ol Z_1\Bigr(\sum^N_{\ell=1}\phi_\ell\Bigr)
=\sum_{\abs{\alpha}+\abs{\beta}+\abs{\gamma}+\gamma_0=N, \alpha, \beta\in\mathbb N^n_0, \gamma_0\in\mathbb N_0, \gamma\in\mathbb N^{d}_0}c_{\alpha,\beta,\gamma_0,\gamma}\ol z^{\alpha} z^{\beta}x^{\gamma_0}_{2n+1}\theta^{\gamma}+O(\abs{x}^{N+1}),
\end{equation}
where $c_{\alpha,\beta,\gamma_0,\gamma}\in\Complex$, for every $\alpha, \beta\in\mathbb N^n_0, \gamma_0\in\mathbb N_0, \gamma\in\mathbb N^{d}_0$. 
Let 
\[\begin{split}
&\phi^{(1)}_{N+1}:=-\sum_{\abs{\alpha}+\abs{\beta}+\abs{\gamma}+\gamma_0=N, \alpha, \beta\in\mathbb N^n_0, \gamma_0\in\mathbb N_0, \gamma\in\mathbb N^{d}_0}c_{\alpha,\beta,\gamma_0,\gamma}\frac{1}{\alpha_1+1}\ol z_1^{\alpha_1+1} \ol z^{\alpha_2}_2\cdots\ol z^{\alpha_n}_nz^{\beta}x^{\gamma_0}_{2n+1}\theta^{\gamma}.\end{split}\]
From \eqref{e-gue180724} and \eqref{e-gue180729}, we see that 
\begin{equation}\label{e-gue180729mp}
\ol Z_1\Bigr(\sum^N_{\ell=1}\phi_\ell+\phi^{(1)}_{N+1}\Bigr)=O(\abs{x}^{N+1}).
\end{equation}

Write 
\begin{equation}\label{e-gue180729I}
\begin{split}
&\ol Z_2\Bigr(\sum^N_{\ell=1}\phi_\ell+\phi^{(1)}_{N+1}\Bigr) =\\
&\sum_{\abs{\alpha}+\abs{\beta}+\abs{\gamma}+\gamma_0=N, \alpha, \beta\in\mathbb N^n_0, \gamma_0\in\mathbb N_0, \gamma\in\mathbb N^{d}_0}d_{\alpha,\beta,\gamma_0,\gamma}\ol z^{\alpha} z^{\beta}x^{\gamma_0}_{2n+1}\theta^{\gamma}+O(\abs{x}^{N+1}),
\end{split}
\end{equation}
where $d_{\alpha,\beta,\gamma_0,\gamma}\in\Complex$, for every $\alpha, \beta\in\mathbb N^n_0, \gamma_0\in\mathbb N_0, \gamma\in\mathbb N^{d}_0$. We claim that 
\begin{equation}\label{e-gue180729II}
\frac{\pr}{\pr\ol z_1}\Bigr(\sum_{\abs{\alpha}+\abs{\beta}+\abs{\gamma}+\gamma_0=N, \alpha, \beta\in\mathbb N^n_0, \gamma_0\in\mathbb N_0, \gamma\in\mathbb N^{d}_0}d_{\alpha,\beta,\gamma_0,\gamma}\ol z^{\alpha} z^{\beta}x^{\gamma_0}_{2n+1}\theta^{\gamma}\Bigr)=0.
\end{equation}
Since $[\ol Z_1, \ol Z_2]\subset C^\infty(X, T^{0,1}X)$, from \eqref{e-gue180729m}, we have 
\begin{equation}\label{e-gue180729mI}
\begin{split}
&[\ol Z_1, \ol Z_2]\Bigr(\sum^N_{\ell=1}\phi_\ell+\phi^{(1)}_{N+1}\Bigr)\\
&=\Bigr(\ol Z_1\ol Z_2-\ol Z_2\ol Z_1\Bigr)\Bigr(\sum^N_{\ell=1}\phi_\ell+\phi^{(1)}_{N+1}\Bigr)\\
&=O(\abs{x}^N). 
\end{split}
\end{equation}
Now, from \eqref{e-gue180729mp}, we have 
\begin{equation}\label{e-gue180729mpI}
\ol Z_2\ol Z_1\Bigr(\sum^N_{\ell=1}\phi_\ell+\phi^{(1)}_{N+1}\Bigr)=O(\abs{x}^N). 
\end{equation}
From \eqref{e-gue180724} and \eqref{e-gue180729I}, we have 
\begin{equation}\label{e-gue180729z}
\begin{split}
&\ol Z_1\ol Z_2\Bigr(\sum^N_{\ell=1}\phi_\ell+\phi^{(1)}_{N+1}\Bigr)\\
&=\frac{\pr}{\pr\ol z_1}\Bigr(\sum_{\abs{\alpha}+\abs{\beta}+\abs{\gamma}+\gamma_0=N, \alpha, \beta\in\mathbb N^n_0, \gamma_0\in\mathbb N_0, \gamma\in\mathbb N^{d}_0}d_{\alpha,\beta,\gamma_0,\gamma}\ol z^{\alpha} z^{\beta}x^{\gamma_0}_{2n+1}\theta^{\gamma}\Bigr)+O(\abs{x}^{N}).
\end{split}
\end{equation}
From \eqref{e-gue180729mI}, \eqref{e-gue180729mpI} and \eqref{e-gue180729z}, we get the claim \eqref{e-gue180729II}. 

Set
\[\begin{split}
&\phi^{(2)}_{N+1}
:=-\sum_{\abs{\alpha}+\abs{\beta}+\abs{\gamma}+\gamma_0=N, \alpha, \beta\in\mathbb N^n_0, \gamma_0\in\mathbb N_0, \gamma\in\mathbb N^{d}_0}d_{\alpha,\beta,\gamma_0,\gamma}\frac{1}{\alpha_2+1}\ol z_1^{\alpha_1} \ol z^{\alpha_2+1}_2\cdots\ol z^{\alpha_n}_nz^{\beta}x^{\gamma_0}_{2n+1}\theta^{\gamma}.\end{split}\]
From \eqref{e-gue180729I} and \eqref{e-gue180729II}, it is not difficult to check that 
\begin{equation}\label{e-gue180730}
\ol Z_j\Bigr(\sum^N_{\ell=1}\phi_\ell+\phi^{(1)}_{N+1}+\phi^{(2)}_{N+1}\Bigr)=O(\abs{x}^{N+1}),\ \ \mbox{for every $j=1, 2$}. 
\end{equation}
Continuing in this way, we conclude that we can find $\phi_{N+1}\in P_{N+1}(U)$ such that
\begin{equation}\label{e-gue180730I}
\ol Z_j\Bigr(\sum^N_{\ell=1}\phi_\ell+\phi_{N+1}\Bigr)=O(\abs{x}^{N+1}),\ \ \mbox{for every $j=1, 2,\ldots,n$}. 
\end{equation}

From \eqref{e-gue180730I} and Mathematical induction, the theorem follows. 
\end{proof}

We can now prove 

\begin{theorem}\label{t-gue180730}
With the notations above, there exists a function $\varphi\in C^\infty(U)$ such that 
\begin{equation}\label{e-gue180730mp}
\mbox{${\rm Im\,}\varphi(x)\geq c\abs{x}^2$ on $U$, where $c>0$ is a constant}, 
\end{equation}
\begin{equation}\label{e-gue180730mpI}
\varphi(x)=x_{2n+1}+i\Bigr(\sum^n_{j=1}\lambda_j\abs{z_j}^2+\abs{x_{2n+1}}^2+\sum^d_{j=1}\abs{\theta_j}^2\Bigr)-\sum^n_{j=1}(\ol c_j\ol z_jx_{2n+1}+c_jz_jx_{2n+1})+O(\abs{x}^3),
\end{equation}
and  
\begin{equation}\label{e-gue180730mpII}
(\ddbar_b\varphi)(x)=O(\abs{x}^N),\ \ \mbox{for every $N\in\mathbb N$},
\end{equation}
where $c_j\in\Complex$, $j=1,\ldots,n$, are as in \eqref{e-gue180724}.
\end{theorem}

\begin{proof}
Let $\phi_\ell\in P_\ell(U)$, $\ell=1,2,3,\ldots$, be as in Theorem~\ref{t-gue180723}. 
For each $j=3,4,\ldots$, we have 
\begin{equation}\label{e2-ysmiI}
\mbox{$\abs{{\rm Im\,} \phi_j(x)}\leq 2^{-j}\Bigr(\sum^n_{j=1}\lambda_j\abs{z_j}^2+\abs{x_{2n+1}}^2+\sum^d_{j=1}\abs{\theta_j}^2\Bigr)$ on $W_j\subset U$},
\end{equation} 
where $W_j$ is an open set of $p$, for each $j=3,4,\ldots$. Take $\chi(x)\in C^\infty_0(\Real^{2n+1+d})$ so that $\chi(x)=1$ if $\abs{x}^2\leq\frac{1}{2}$ 
and $\chi(x)=0$ if $\abs{x}^2\geq1$. For each $j=3,4,\ldots$, take $\epsilon_j>0$ be a small constant so that 
${\rm Supp\,}\chi(\frac{x}{\epsilon_j})\subset W_j$, 
\begin{equation}\label{e-gue180730aI}
\left| \chi(\frac{x}{\epsilon_j}){\rm Im\,}\phi_j(x) \right| <2^{-j}\Bigr(\sum^n_{j=1}\lambda_j\abs{z_j}^2+\abs{x_{2n+1}}^2+\sum^d_{j=1}\abs{\theta_j}^2\Bigr)\ \ \mbox{on $U$},
\end{equation}
and for all $\alpha\in\mathbb N^{2n+1+d}_0$, $\abs{\alpha}<j$, we have 
\begin{equation}\label{e-gue180730a}
\sup\set{\abs{\pr^\alpha_x\bigr(\chi(\frac{x}{\epsilon_j})\phi_j(x)\bigr)};\, x\in U} <2^{-j}.
\end{equation} 
On $U$, we put
\[\varphi(x)=\phi_1(x)+\phi_2(x)+\sum^{\infty}_{j=3}\chi(\frac{x}{\epsilon_j})\phi_j(x).\] 
From \eqref{e-gue180730a}, we can check that $\varphi(x)$ is well-defined as a smooth function on $U$ and 
for all $\alpha\in N^{2n+1+d}_0$ with $\abs{\alpha}=k$, $k\in\mathbb N$, we have 
\[\pr^\alpha_{x}\varphi|_{p}=\pr^\alpha_{x}\phi_k|_{p}.\]
Combining this with \eqref{e-gue180724II}, we conclude that $\ddbar_b\varphi$ vanishes to infinite order at $p$ and we get \eqref{e-gue180730mpII}. 
 Moreover, from \eqref{e-gue180730aI}, we have 
\[\begin{split}
&{\rm Im\,}\varphi(x)
\geq\Bigr(\sum^n_{j=1}\lambda_j\abs{z_j}^2+\abs{x_{2n+1}}^2+\sum^d_{j=1}\abs{\theta_j}^2\Bigr)\Bigr(1-\sum^{\infty}_{j=3}2^{-j}\Bigr)\geq c\abs{x}^2\ \ \mbox{on $U$},\end{split}\]
where $c>0$ is a constant. We get \eqref{e-gue180730mp}. From the construction of $\varphi$, we see that \eqref{e-gue180730mpI} holds and the theorem follows. \end{proof}

Put 
\begin{equation}\label{e-gue180730b}
N_p:=\set{g\in G;\, g\circ p=p}=\set{g_1:=e_0, g_2,\ldots,g_r},
\end{equation}
where $g_j\neq g_k$, if $j\neq k$, $j, k=1,\ldots,r$. We need 

\begin{lemma}\label{l-gue180730b}
There are small open sets $W\Subset V$ of $e_0$ in $G$ and $D\subset U$ of $p$ in $X$ such that 
\[\mbox{$g\circ x\in U$, for every $x\in D$, for every $g\in\bigcup^r_{s, t=1}g_sWg_t$},\]
where $V$ and $U$ are as in the discussion before \eqref{e-gue180722I}. 
\end{lemma}

\begin{proof}
This follows directly from the fact that the action map $G \times X \rightarrow X$ is continuous.
\end{proof}

For any set $A\subset G$, put $A^{-1}:=\set{g^{-1};\, g\in G}$. If $A$ is an open set of $G$, we have that $A^{-1}$ is also open in $G$ because taking inverse is continuous. Let $W$ be as in Lemma~\ref{l-gue180730b}. We can replace $W$ by $W\bigcap W^{-1}$.   
Fix $W_0\Subset W_1\Subset W$, where $W_0$ and $W_1$ are open neighborhoods of $e_0$ in $G$. 
From now on, we take $W_0$, $W_1$, $W$ small enough so that 
\begin{equation}\label{e-gue180731u}
\begin{split}
&(g_sWg_t)\bigcap(g_{s_1}Wg_{t_1})=\emptyset,\ \ \mbox{for every $s, t, s_1, t_1=1,\ldots,r$, with $g_s\circ g_t\neq g_{s_1}\circ g_{t_1}$},\\
&W=W^{-1},\ \ W_0=W_0^{-1},\  \ W_1=W_1^{-1}. 
\end{split}
\end{equation}
We need

\begin{lemma}\label{l-gue180730bI}
There is an open set $\hat D\Subset D$ of $p$ in $X$ such that 
\[\mbox{$g\circ x\notin\hat D$, for every $x\in\hat D$ and every $g\notin\bigcup^r_{s=1}W_1g_s$},\]
where $D$ is as in Lemma~\ref{l-gue180730b}. 
\end{lemma}

\begin{proof}
Assume that the claim of the lemma is not true. Then, we can find open sets  $D_j\subset D$ of $p$ in $X$, $j=1,2,\ldots$, with $D_{j+1}\subset D_j$, $j=1,2,\ldots$, $\bigcap^{+\infty}_{j=1}D_j=\set{p}$, such that for every $j$, we can find 
$h_j\notin\bigcup^r_{s=1}W_1g_s$ and $x_j\in D_j$, such that $h_j\circ x_j\in D_j$. 
Since $G$ is compact, we can find a subsequence $1\leq t_1<t_2<\cdots$, such that $h_{t_\ell}\To h\in G$ as $t_\ell\To+\infty$, for some 
\[h\notin\ol{\bigcup^r_{s=1}W_0g_s}.\]
 Since $\bigcap^{+\infty}_{j=1}D_j=\set{p}$, we have $x_{t_\ell}\To p$ as $t_\ell\To+\infty$. Since $h_{t_\ell}\circ x_{t_\ell}\in D_{t_\ell}$, for every $t_\ell$, we conclude that 
 $h\circ p=p$. But $h\notin N_p$, we get a contradiction. The lemma follows. 
\end{proof}

We pause and introduce some notations. Let $\Omega$ be an open set of $\Real^N$, $N\in\mathbb N$. Let $g_k(x)\in C^\infty(\Omega)$ be a $k$-dependent function, $k=1,2,\ldots$. We write $g_k=O(k^{-\infty})$ if for every $M\in\mathbb N$, every compact set $K\Subset\Omega$ and every $\alpha\in\mathbb N_0^N$, there is a constant $C>0$ independent of $k$ such that
\[\abs{(\pr^\alpha_xg_k)(x)}\leq Ck^{-M},\ \ \mbox{for every $x\in K$ and every $k=1,2,\ldots$}.\]
Let $g_k(x), h_k(x)\in C^\infty(\Omega)$ be $k$-dependent functions, $k=1,2,\ldots$. We write $g_k=h_k+O(k^{-\infty})$ if $g_k-h_k=O(k^{-\infty})$. 

Let $\chi\in C^\infty_0(\hat D)$ with $\chi=1$ near $p$, where $\hat D$ is as in Lemma~\ref{l-gue180730bI}. Let $\tau\in C^\infty_0(\Real^{d})$ with $\tau=1$ near $0\in\Real^{d}$. Put $x''=(x_{2n+2},\ldots,x_{2n+1+d})$. For every $k\in\mathbb N$, put 
\begin{equation}\label{e-gue180730d}
\Td f_k(x)=k^{\frac{d}{2}}e^{ik\varphi(x)}\chi(x)\tau(\frac{\sqrt{k}}{\log k}x'')\in C^\infty_0(\hat D). 
\end{equation}
From \eqref{e-gue180730mp} and \eqref{e-gue180730mpII}, we see that 
\begin{equation}\label{e-gue180731p}
(\ddbar_b\Td f_k)(x)=O(k^{-\infty}). 
\end{equation}
Let $a(g)\in C^\infty(G)$ be a smooth function on $G$ so that ${\rm Re\,}\ol{a(g)}=2$ on some small neighborhood of $\bigcup^r_{s, t=1}g_sW_0g_t$ and $\abs{a(g)}=0$ outside $\bigcup^r_{s,t=1}g_sW_1g_t$. By Theorem~\ref{t-gue180710mp}, 
we see that there is a $b(g)\in C^\infty_G(G)$ so that $\norm{b-a}_{C^0(G)}<\frac{1}{2}$. Hence, ${\rm Re\,}\ol{b(g)}\geq 1$ on some small neighborhood of $\bigcup^r_{s,t=1}g_sW_0g_t$ and $\abs{b(g)}\leq\frac{1}{2}$ outside 
$\bigcup^r_{s,t=1}g_sW_1g_t$. For every $k\in\mathbb N$, put 
\begin{equation}\label{e-gue180801m}
\hat f_k(x):=\int_G\Td f_k(g\circ x)\ol{b(g)}d\mu(g)\in C^\infty_G(X).
\end{equation}
Let 
\[S: L^2(X)\To{\rm Ker\,}\ddbar_b:=\set{u\in L^2(X);\, \ddbar_bu=0}\] 
be the orthogonal projection. 

\begin{theorem}\label{t-gue180731r}
We have $S\hat f_k\in C^\infty_G(X)\bigcap{\rm Ker\,}\ddbar_b$ and 
\begin{equation}\label{e-gue180731s}
S\hat f_k=\hat f_k+O(k^{-\infty}).
\end{equation}
\end{theorem}

\begin{proof}
Since $\hat f_k\in C^\infty_G(X)$, we have $\hat f_k=\sum^N_{j=1}g_{j,k}$, $N\in\mathbb N$, where $g_{j,k}\in C^\infty_{m_j}(X)$, for some $m_j\in\mathbb N$, $j=1,\ldots,N$. Note that $N$ and $m_j$, $j=1,\ldots,N$, are independent of $k$. From \eqref{e-gue180731p}, we see that 
\begin{equation}\label{e-gue180731sI}
\ddbar_bg_{j,k}=O(k^{-\infty}), \ \ \mbox{for every $j=1,\ldots,N$}. 
\end{equation}
It is easy to see that 
\begin{equation}\label{e-gue180731sII}
S\hat f_k=\sum^N_{j=1}S_{m_j}g_{j,k},
\end{equation}
where $S_m$ is given by \eqref{e-gue180731w}. From \eqref{e-gue180731sI}, \eqref{e-gue180731sII}, \eqref{e-gue180720f} and \eqref{e-gue180720fI}, the theorem follows. 
\end{proof}

For every $k\in\mathbb N$, put $f_k:=S\hat f_k\in C^\infty_G(X)\bigcap{\rm Ker\,}\ddbar_b$ and set $G\hat D:=\set{g\circ x;\, g\in G, x\in\hat D}$. 
 
\begin{theorem}\label{t-gue180731w}
We have 
\begin{equation}\label{e-gue180731z}
\begin{split}
&\lim_{k\To+\infty}f_k(p)=\pi^{\frac{d}{2}}\sum^r_{j=1}\ol{b(g_j)},\\
&{\rm Re\,}(\lim_{k\To+\infty}f_k(p))\geq r\pi^{\frac{d}{2}}
\end{split}
\end{equation}
and 
\begin{equation}\label{e-gue180731zw}
\lim_{k\To+\infty}\abs{f_k(h\circ p)}=\pi^{\frac{d}{2}}\abs{\sum^r_{j=1}\ol{b(g_j\circ h^{-1})}}\leq\frac{1}{2}r\pi^{\frac{d}{2}} \ \ \mbox{for every $h\notin\bigcup^r_{s,t=1}g_sW_1g_t$},
\end{equation}
where $r\in\mathbb N$ is the Cardinal number of $N_p$. 

For every $\varepsilon>0$, there is a $k_1\in\mathbb N$ such that 
\begin{equation}\label{e-gue180802q}
\abs{f_k(x)}\leq\varepsilon,\ \ \mbox{for every $k\geq k_1$, $k\in\mathbb N$, and every $x\notin G\hat D$}. 
\end{equation}
\end{theorem}

\begin{proof}
Take $\gamma\in C^\infty_0(W)$ with $\gamma=1$ on $W_1$. From \eqref{e-gue180731s} and \eqref{e-gue180801m}, we have 
\begin{equation}\label{e-gue180801}
\begin{split}
&\lim_{k\To+\infty}f_k(p)\\
&=\lim_{k\To+\infty}\hat f_k(p)\\
&=\lim_{k\To+\infty}\int_G\Td f_k(g\circ p)\ol{b(g)}d\mu(g)\\
&=\lim_{k\To+\infty}\Bigr(\int_G\Td f_k(g\circ p)\sum^r_{j=1}\gamma(g\circ g^{-1}_j)\ol{b(g)}d\mu(g)\\
&\quad\quad+\int_G\Td f_k(g\circ p)\bigr(1-\sum^r_{j=1}\gamma(g\circ g^{-1}_j)\bigr)\ol{b(g)}d\mu(g)\Bigr). 
\end{split}
\end{equation}
If $g\in W_1g_s$, for some $s=1,\ldots,r$, then $1-\sum^r_{j=1}\gamma(g\circ g^{-1}_j)=0$. From this observation, we deduce that 
\begin{equation}\label{e-gue180801mI}
\int_G\Td f_k(g\circ p)\bigr(1-\sum^r_{j=1}\gamma(g\circ g^{-1}_j)\bigr)\ol{b(g)}d\mu(g)
=\int_{g\notin\bigcup^r_{s=1}W_1g_s}\Td f_k(g\circ p)\bigr(1-\sum^r_{j=1}\gamma(g\circ g^{-1}_j)\bigr)\ol{b(g)}d\mu(g).
\end{equation}
From Lemma~\ref{l-gue180730bI}, we see that 
\begin{equation}\label{e-gue180801mII}
g\circ p\notin\hat D,\ \ \mbox{for every $g\notin\bigcup^r_{s=1}W_1g_s$}.
\end{equation}
By the construction of  $\Td f_k(x)$ (see \eqref{e-gue180730d}),  $\Td f_k(x)=0$, for every $x\notin\hat D$. From this observation, \eqref{e-gue180801mI} and \eqref{e-gue180801mII}, we conclude that 
\begin{equation}\label{e-gue180801mIII}
\int_G\Td f_k(g\circ p)\bigr(1-\sum^r_{j=1}\gamma(g\circ g^{-1}_j)\bigr)\ol{b(g)}d\mu(g)=0.
\end{equation}

Now, from \eqref{e-gue180801mIII}, \eqref{e-gue180801}, \eqref{e-gue180730d} and \eqref{e-gue180722I}, we have 
\begin{equation}\label{e-gue180801p}
\begin{split}
&\lim_{k\To+\infty}f_k(p)\\
&=\lim_{k\To+\infty}\int_G\Td f_k(g\circ p)\sum^r_{j=1}\gamma(g\circ g^{-1}_j)\ol{b(g)}d\mu(g)\\
&=\lim_{k\To+\infty}\int_G\Td f_k(g\circ g_j\circ p)\sum^r_{j=1}\gamma(g)\ol{b(g\circ g_j)}d\mu(g)\\
&=\lim_{k\To+\infty}\int_G\Td f_k(g\circ p)\sum^r_{j=1}\gamma(g)\ol{b(g\circ g_j)}d\mu(g)\\
&=\lim_{k\To+\infty}\sum^r_{j=1}\int e^{ik\varphi(0,\theta)}\chi((0,\theta))\tau(\frac{\sqrt{k}}{\log k}\theta)\ol{b(\theta\circ g_j)}m(\theta)d\theta\\
&=\lim_{k\To+\infty}\sum^r_{j=1} \int e^{-\abs{\theta}^2}\chi((0,\frac{\theta}{\sqrt{k}}))\tau(\frac{\theta}{\log k})\ol{b(\frac{\theta}{\sqrt{k}}\circ g_j)}m(\frac{\theta}{\sqrt{k}})d\theta\\
&=\sum^r_{j=1}\int_{\Real^d}e^{-\abs{\theta}^2}\ol{b(g_j)}d\theta\\
&=\pi^{\frac{d}{2}}\sum^r_{j=1}\ol{b(g_j)},
\end{split}
\end{equation}
where $(0,\theta)=(0,\ldots,0,\theta_1,\ldots,\theta_d)\in U$, $d\mu(\theta)=m(\theta)d\theta$ on $W$, $m(0)=1$. Note that ${\rm Re\,}\ol{b(g_j)}\geq 1$, for every $j=1,\ldots,r$. From this observation and \eqref{e-gue180801p}, we get \eqref{e-gue180731z}. 

Now, for every $h\in G$, we have 
\begin{equation}\label{e-gue180801y}
\begin{split}
&\lim_{k\To+\infty}\abs{f_k(h\circ p)}\\
&=\lim_{k\To+\infty}\abs{\hat f_k(h\circ p)}\\
&=\lim_{k\To+\infty}\abs{\int_G\Td f_k(g\circ h\circ p)\ol{b(g)}d\mu(g)}\\
&=\lim_{k\To+\infty}\abs{\int_G\Td f_k(g\circ p)\ol{b(g\circ h^{-1})}d\mu(g)}.
\end{split}
\end{equation}
From \eqref{e-gue180801y}, we can repeat the process in  \eqref{e-gue180801p} and conclude that 
\begin{equation}\label{e-gue180801yI}
\lim_{k\To+\infty}\abs{f_k(h\circ p)}=\pi^{\frac{d}{2}}\abs{\sum^r_{j=1}\ol{b(g_j\circ h^{-1})}}.
\end{equation}
Assume that $h\notin\bigcup^r_{s,t}g_sW_1g_t$. Note that $W_1=W^{-1}_1$. From this observation, it is no difficult to check that 
\begin{equation}\label{e-gue180801yII}
g_j\circ h^{-1}\notin\bigcup^r_{s,t}g_sW_1g_t,\ \ \mbox{for every $j=1,\ldots,r$}. 
\end{equation}
Note that $\abs{\ol{b(g)}}\leq\frac{1}{2}$, for every $g\notin \bigcup^r_{s,t}g_sW_1g_t$. From this observation, \eqref{e-gue180801yII} and \eqref{e-gue180801yI}, we get \eqref{e-gue180731zw}. 

Let $\varepsilon>0$ be positive. From \eqref{e-gue180731s}, we see that there is a $k_1\in\mathbb N$ such that for every $k\in\mathbb N$, $k\geq k_1$, we have 
\begin{equation}\label{e-gue180806mp}
\abs{f_k(x)-\hat f_k(x)}<\varepsilon,\ \ \mbox{for every $x\in X$}.
\end{equation}
Now, for $x\notin G\hat D$, we have 
\[\int_G\Td f_k(g\circ x)\ol{b(g)}d\mu(g)=0\]
since ${\rm Supp\,}\Td f_k\subset\hat D$. From this observation and \eqref{e-gue180806mp}, we get \eqref{e-gue180802q}. 
\end{proof}

From Theorem~\ref{t-gue180731w}, we deduce 

\begin{corollary}\label{c-gue180802}
With the notations used above, there is a $k_0\in\mathbb N$ and an open set $\Td D\Subset\hat D$ of  $p$ such that
\begin{equation}\label{e-gue180802qa}
{\rm Re\,}f_{k_0}(x)\geq\frac{3}{4}r\pi^{\frac{d}{2}},\  \ \mbox{for every $x\in\Td D$},
\end{equation}
and 
\begin{equation}\label{e-gue180802qaIz}
\abs{f_{k_0}(h\circ x)}\leq \frac{2}{3}r\pi^{\frac{d}{2}} \ \ \mbox{for every $h\notin\bigcup^r_{s,t=1}g_sW_1g_t$ and every $x\in\Td D$},
\end{equation}
where $r\in\mathbb N$ is the Cardinal number of $N_p$. 

Moreover, let $x_0\in X$ and assume that $g\circ x_0\notin\hat D$, for every $g\in G$. Then, 
\begin{equation}\label{e-gue180802qaII}
\abs{f_{k_0}(x_0)}\leq \frac{1}{2}r\pi^{\frac{d}{2}}.
\end{equation}
\end{corollary}

We can repeat the proof of Theorem~\ref{t-gue180730} with minor change and get

\begin{theorem}\label{t-gue180801s}
We can find $\beta_j(x)\in C^\infty(U)$, $j=1,2,\ldots,n+1$, with 
\begin{equation}\label{e-gue180801s}
\begin{split}
&\beta_j(x)=z_j+O(\abs{x}^2),\ \ j=1,2,\ldots,n,\\
&\beta_{n+1}(x)=x_{2n+1}+O(\abs{x}^2)
\end{split}
\end{equation}
such that $\ddbar_b\beta_j(x)=O(\abs{x}^N)$, for every $N\in\mathbb N$ and every $j=1,2,\ldots,n+1$. 
\end{theorem}

For every $j=1,\ldots,n+1$, and every $k\in\mathbb N$, put 
\begin{equation}\label{e-gue180801sI}
\Td f^{(j)}_k(x):=k^{\frac{d}{2}}e^{ik\varphi(x)}\beta_j(x)\chi(x)\tau(\frac{\sqrt{k}}{\log k}x'')\in C^\infty_0(\hat D),
\end{equation}
set 
\[\hat f^{(j)}_k(x):=\int_G\Td f^{(j)}_k(g\circ x)\ol{b(g)}d\mu(g)\in C^\infty_G(X)\]
and 
put
\[f^{(j)}_k:=S\hat f^{(j)}_k.\]
We can repeat the proofs of Theorem~\ref{t-gue180731r} and \eqref{e-gue180731z} and obtain 

\begin{theorem}\label{t-gue180802}
We have $f^{(j)}_k\in C^\infty_G(X)\bigcap{\rm Ker\,}\ddbar_b$, $j=1,2,\ldots,n+1$, and 
\begin{equation}\label{e-gue180802}
\begin{split}
&\lim_{k\To+\infty}df^{(j)}_k(p)=\pi^{\frac{d}{2}}\sum^r_{s=1}\ol{b(g_s)}dz_j,\ \ j=1,2,\ldots,n,\\
&\lim_{k\To+\infty}df^{(n+1)}_k(p)=\pi^{\frac{d}{2}}\sum^r_{s=1}\ol{b(g_s)}dx_{2n+1}.
\end{split}
\end{equation}
\end{theorem}

Let $\xi_1,\ldots,\xi_d$ be the orthonormal basis for $\mathfrak{g}$ as in the discussion before \eqref{e-gue180713}. For every $j=1,2,\ldots,d$, put 
\begin{equation}\label{e-gue180802I}
\begin{split}
\Td T_j: C^\infty(G)&\To C^\infty(G),\\
u(g)&\To\frac{\pr}{\pr t}\Bigr(u(g\circ (e^{t\xi_j})^{-1})\Bigr)|_{t=0}.
\end{split}
\end{equation}
It is easy to see that $\Td T_j$, $j=1,\ldots,d$, are linear independent vector fields on $G$. By using local coordinates, it is straightforward to see that there are 
$\delta_\ell(g)\in C^\infty(G)$, $\ell=1,\ldots,d$, such that 
\begin{equation}\label{e-gue180802IIp}
\mbox{the matrix $\left(\sum^r_{s=1}(\Td T_j\ol{\delta_\ell})(g_s)\right)^d_{j,\ell=1}$ is invertible.}
\end{equation}
We remind the reader that $g_1\in G,\ldots,g_r\in G$ are as in \eqref{e-gue180730b}. Let $0<\varepsilon<1$ be a constant. By Theorem~\ref{t-gue180710mp}, we can find 
$\alpha_\ell(g)\in C^\infty_G(G)$, $\ell=1,\ldots,d$, such that $\norm{a_\ell-\delta_\ell}_{C^1(G)}<\varepsilon$, for every $\ell=1,\ldots,d$. We take $\varepsilon$ small enough so that 
\begin{equation}\label{e-gue180802II}
\mbox{the matrix $\left(\sum^r_{s=1}(\Td T_j\ol{\alpha_\ell})(g_s)\right)^d_{j,\ell=1}$ is invertible.}
\end{equation}
For every $\ell=1,\ldots,d$, and every $k\in\mathbb N$, put 
\begin{equation}\label{e-gue180801sIp}
\hat H^{(\ell)}_k(x):=\int_G\Td f_k(g\circ x)\ol{\alpha_\ell(g)}d\mu(g)\in C^\infty_G(X),
\end{equation}
where $\Td f_k(x)$ is as in \eqref{e-gue180730d}. Put 
\[H^{(\ell)}_k:=S\hat H^{(\ell)}_k.\]
As before, we have 
$H^{(\ell)}_k\in C^\infty_G(X)\bigcap{\rm Ker\,}\ddbar_b$, $\ell=1,2,\ldots,d$. Let $T_j\in C^\infty(X,TX)$, $j=1,\ldots,d$, be as in \eqref{e-gue180713}.

\begin{theorem}\label{t-gue180802m}
We have
\begin{equation}\label{e-gue180802r}
\lim_{k\To+\infty}(T_jH^{(\ell)}_k)(p)=\pi^{\frac{d}{2}}\sum^r_{s=1}(\Td T_j\ol{\alpha_\ell})(g_s),\ \ j, \ell=1,2,\ldots,d,
\end{equation}
$\lim_{k\To+\infty}(dH^{(\ell)}_k)(p)$ exists, for every $\ell=1,2,\ldots,d$, and 
\begin{equation}\label{e-gue180802rI}
\mbox{$\set{\lim_{k\To+\infty}(dH^{(\ell)}_k)(p);\, \ell=1,2,\ldots,d}$ are linear independent}. 
\end{equation}
\end{theorem}

\begin{proof}
For every $j, \ell=1,2,\ldots,d$, we have 
\begin{equation}\label{e-gue180802rII}
\begin{split}
&\lim_{k\To+\infty}(T_jH^{(\ell)}_k)(p)\\
&=\lim_{k\To+\infty}(T_j\hat H^{(\ell)}_k)(p)\\
&=\lim_{k\To+\infty}\frac{\pr}{\pr t}\Bigr(\int_G\Td f_k(g\circ e^{t\xi_j}\circ p)\ol{\alpha_\ell(g)}d\mu(g)\Bigr)|_{t=0}\\
&=\lim_{k\To+\infty}\frac{\pr}{\pr t}\Bigr(\int_G\Td f_k(g\circ p)\ol{\alpha_\ell(g\circ(e^{t\xi_j})^{-1})}d\mu(g)\Bigr)|_{t=0}\\
&=\lim_{k\To+\infty}\int_G\Td f_k(g\circ p)(\Td T_j\ol{\alpha_\ell})(g)d\mu(g).
\end{split}
\end{equation}
From \eqref{e-gue180802rII}, we can repeat the proof of \eqref{e-gue180731z} and conclude that $\lim_{k\To+\infty}(dH^{(\ell)}_k)(p)$ exists, for every $\ell=1,2,\ldots,d$, and 
\eqref{e-gue180802r} hold. 

From \eqref{e-gue180802r}, \eqref{e-gue180802II} and some elementary linear algebra, we get \eqref{e-gue180802rI}. 
\end{proof}

From \eqref{e-gue180802} and \eqref{e-gue180802rI}, we conclude that there is a $k_0\in\mathbb N$, such that for every $k\geq k_0$, 
\begin{equation}\label{e-gue180802ra}
\begin{split}
&\set{(df^{(1)}_k)(p),(df^{(2)}_k)(p),\ldots,(df^{(n+1)}_k(p), (dH^{(1)}_k)(p), (dH^{(2)}_k)(p),\ldots, (dH^{(d)}_k)(p)}\\
&\mbox{are linear independent}.
\end{split}
\end{equation}
Put 
\begin{equation}\label{e-gue180802raI}
\begin{split}
&f^{(j)}_p:=f^{(j)}_{k_0}\in C^\infty_G(X)\bigcap{\rm Ker\,}\ddbar_b,\  \ j=1,2,\ldots,n+1,\\
&H^{(\ell)}_p:=H^{(\ell)}_{k_0}\in C^\infty_G(X)\bigcap{\rm Ker\,}\ddbar_b,\ \ \ell=1,2,\ldots,d.
\end{split}
\end{equation}
Consider the $G$-equivariant CR map
\begin{equation}\label{e-gue180802raII}
\begin{split}
\hat F_p: X&\To\Complex^{n+1+d},\\
x&\To (f^{(1)}_p(x),f^{(2)}_p(x),\ldots,f^{(n+1)}_p(x), H^{(1)}_p(x),H^{(2)}_p(x),\ldots,H^{(d)}_p(x))\in\Complex^{n+1+d}.
\end{split}
\end{equation}
From \eqref{e-gue180802ra}, we get 

\begin{theorem}\label{t-gue180802w}
With the notations above, there is an open set $D_p\subset U$ of $p$ in $X$ such that the differential of the map $\hat F_p$ is injective at every point of $D_p$ and $\hat F_p$ is injective on $D_p$. 
\end{theorem}

From Theorem~\ref{t-gue180802w}, we see that $X$ is locally $G$-equivariant embeddable by a CR map. To get global $G$-equivariant embedding, we need more $G$-finite smooth CR functions. We can repeat the proofs of Lemma~\ref{l-gue180730b} and Corollary~\ref{c-gue180802} and obtain

\begin{theorem}\label{t-gue180802ta}
With the notations used above, there are small open sets $W_p\Subset V$ of $e_0$ in $G$ with $W_p=W^{-1}_p$, $\Td D_{p}\Subset D_p$ of $p$ in $X$  and  a $G$-finite smooth CR function $f_p\in C^\infty_G(X)\bigcap{\rm Ker\,}\ddbar_b$ such that 
\begin{equation}\label{e-gue180802yn}
\mbox{$g\circ x\in D_{p}$, for every $x\in\Td D_p$ and every $g\in\bigcup^r_{s, t=1}g_sW_pg_t$},
\end{equation}
\begin{equation}\label{e-gue180802ynI}
{\rm Re\,}f_{p}(x)\geq\frac{3}{4}r\pi^{\frac{d}{2}},\  \ \mbox{for every $x\in\Td D_p$},
\end{equation}
and 
\begin{equation}\label{e-gue180802qaI}
\abs{f_{p}(h\circ x)}\leq \frac{2}{3}r\pi^{\frac{d}{2}} \ \ \mbox{for every $h\notin\bigcup^r_{s,t=1}g_sW_pg_t$ and every $x\in\Td D_p$},
\end{equation}
where $r\in\mathbb N$ is the Cardinal number of $N_p$ and $D_p$ is an open set of $p$ as in Theorem~\ref{t-gue180802w}. 
\end{theorem}

Similarly, we can repeat the proof of Corollary~\ref{c-gue180802} and get 

\begin{theorem}\label{t-gue180802taI}
With the notations used above, there is an open set $D_{0,p}\Subset\Td D_p$ of $p$ in $X$ and a $G$-finite smooth CR function $A_{p}\in C^\infty_G(X)\bigcap{\rm Ker\,}\ddbar_b$ such that 
\begin{equation}\label{e-gue180802ta}
{\rm Re\,}A_{p}(x)\geq\frac{3}{4}r\pi^{\frac{d}{2}},\  \ \mbox{for every $x\in D_{0,p}$},
\end{equation}
and for every $x\in X$ with $g\circ x\notin\Td D_p$, for every $g\in G$, we have 
\begin{equation}\label{e-gue180802taI}
\abs{A_p(x)}\leq \frac{1}{2}r\pi^{\frac{d}{2}},
\end{equation}
where $\Td D_p$ is an open set of $p$ as in Theorem~\ref{t-gue180802ta}.
\end{theorem}

For every $x\in X$, let $D_{x}$, $\Td D_x$ and $D_{0,x}$ be open sets as in Theorem~\ref{t-gue180802w}, Theorem~\ref{t-gue180802ta} and Theorem~\ref{t-gue180802taI} respectively and let $f^{(j)}_x\in C^\infty_G(X)\bigcap{\rm Ker\,}\ddbar_b$, $j=1,2,\ldots,n+1$, $H^{(\ell)}_x\in C^\infty_G(X)\bigcap{\rm Ker\,}\ddbar_b$, $\ell=1,\ldots,d$, $f_x\in C^\infty_G(X)\bigcap{\rm Ker\,}\ddbar_b$, $A_x\in C^\infty_G(X)\bigcap{\rm Ker\,}\ddbar_b$ be as in \eqref{e-gue180802raI}, Theorem~\ref{t-gue180802ta} and Theorem~\ref{t-gue180802taI}. Suppose that 
\[X=D_{0,p_1}\bigcup D_{0,p_2}\bigcup\cdots\bigcup D_{0,p_K},\ \ K\in\mathbb N.\]
For every $p_j$, $j=1,2,\ldots,K$, put 
\begin{equation}\label{e-gue180803}
\begin{split}
F_{p_j}: X&\To\Complex^{n+3+d},\\
x&\To (f^{(1)}_{p_j}(x),\ldots,f^{(n+1)}_{p_j}(x),H^{(1)}_{p_j}(x),\ldots,H^{(d)}_{p_1}(x),f_{p_j}(x),A_{p_j}(x))\in\Complex^{n+3+d}.
\end{split}
\end{equation}
Consider the $G$-equivariant CR map
\begin{equation}\label{e-gue180803I}
\begin{split}
F: X&\To\Complex^{K(n+3+d)},\\
x&\To (F_{p_1}(x), F_{p_2}(x),\ldots,F_{p_K}(x))\in\Complex^{K(n+3+d)}.
\end{split}
\end{equation}
We can now prove our main result

\begin{theorem}\label{t-gue180803}
With the notations used above, $F$ is an embedding.
\end{theorem}

\begin{proof}
From Theorem~\ref{t-gue180802w}, we see that the differential of $F$ is injective at every point of $X$. Hence, to prove the theorem, we only need to show that $F$ is globally injective. Fix $x_0\in X$ and $y_0\in X$ with $x_0\neq y_0$. We are going to prove that $F(x_0)\neq F(y_0)$. We may assume that $x_0\in D_{0,p_1}$. 

\begin{itemize}
\item[case I.] If $g\circ y_0\notin\Td D_{p_1}$, for every $g\in G$ : From \eqref{e-gue180802ta} and \eqref{e-gue180802taI}, we see that 
\[\abs{A_{p_1}(x_0)}\geq\frac{3}{4}r\pi^{\frac{d}{2}}>\frac{1}{2}r\pi^{\frac{d}{2}}\geq\abs{A_{p_1}(y_0)}.\]
Hence, $F(x_0)\neq F(y_0)$.

\item[case II.] If $g\circ y_0\in\Td D_{p_1}$, for some $g\in\bigcup^r_{s, t=1}g_sW_{p_1}g_t$, where $W_{p_1}$ is as in Theorem~\ref{t-gue180802ta} : Take $g_0\in\bigcup^r_{s, t=1}g_sW_{p_1}g_t$ so that $g_0\circ y_0=y_1\in\Td D_{p_1}$. Then, $y_0=g^{-1}_0\circ y_1$. Since $W_{p_1}^{-1}=W_{p_1}$, we can check that $g^{-1}_0\in\bigcup^r_{s, t=1}g_sW_{p_1}g_t$. From this observation and \eqref{e-gue180802yn}, we conclude that $y_0=g^{-1}_0\circ y_1\in D_{p_1}$. In view of Theorem~\ref{t-gue180802w}, we see that $\hat F_{p_1}$ is injective on $D_{p_1}$ and hence $F$ is injective on $D_{p_1}$. Thus, $F(x_0)\neq F(y_0)$.

\item[case III.] If $g\circ y_0\in\Td D_{p_1}$, for some $g\notin\bigcup^r_{s, t=1}g_sW_{p_1}g_t$ : Take $h\notin\bigcup^r_{s, t=1}g_sW_{p_1}g_t$ so that $h\circ y_0=y_1\in\Td D_{p_1}$. Since $W_{p_1}^{-1}=W_{p_1}$, we have $h^{-1}\notin\bigcup^r_{s, t=1}g_sW_{p_1}g_t$. From this observation, \eqref{e-gue180802ynI} and \eqref{e-gue180802qaI}, we have
\[\abs{f_{p_1}(x_0)}\geq\frac{3}{4}r\pi^{\frac{d}{2}}>\frac{2}{3}r\pi^{\frac{d}{2}}\geq\abs{f_{p_1}(h^{-1}\circ y_1)}=\abs{f_{p_1}(y_0)}.\]
Hence, $F(x_0)\neq F(y_0)$. 
\end{itemize}

We have proved that $F(x_0)\neq F(y_0)$. The theorem follows. 
\end{proof}

\section{CR Orbifolds}\label{Sec:CROrbifolds}

\subsection{Definition and properties of CR orbifolds}\label{s-gue181015}

We start by giving the basic definitions.

\begin{definition}\label{d-gue181015}
Let $X$ be a Hausdorff topological space. 
We say that $X$ is a CR orbifold of dimension $2n+d$ with CR codimension $d$ if there exists a cover $U_i$ of $X$, which is closed under finite intersections, such that
\begin{itemize}
\item For each $U_i$, there exists a CR manifold $V_i$ of dimension $2n+d$ with CR codimension $d$, a finite group $\Gamma_i$ acting on $V_i$ with CR automorphisms and a $\Gamma_i$-invariant map $\Psi_i \colon V_i \rightarrow U_i$, which induces a homeomorphism $V_i / \Gamma_i \rightarrow U_i$. 
\item For each inclusion $U_i \subset U_j$, we have an injective group morphism $\varphi_{ij} \colon \Gamma_i \rightarrow \Gamma_j$ and a CR isomorphism $\Phi_{ij} \colon V_i \rightarrow \Psi_j^{-1}(U_i)$, which satisfies $\Phi_{ij}(gx) = \varphi_{ij} (g) \Phi_{ij} (x)$ for $x \in V_i$, $g \in \Gamma_i$ and fulfills $\Psi_j \circ \Phi_{ij} = \Psi_i$.
\end{itemize}
We call the tuple $(U_i, V_i, \Gamma_i, \Psi_i)$ an orbifold chart and the cover $U_i$ an orbifold atlas.

An orbifold is called effective if for all charts, the action of $\Gamma_i$ on $V_i$ is effective.

For $U \subset X$ open, a function $f \colon U \rightarrow \C$ is called a CR function if every lift of $f$ into a chart is a CR function.
\end{definition}

It is a well known fact that every real effective orbifold may be written as the global quotient of a compact group $G$ acting locally free on a manifold $X$.
We will formulate the necessary conditions to proof the according theorem for CR orbifolds.\\

Let $G$ be a Lie-group, $H < G$ a closed Lie-subgroup of $G$ and $H$ act on a real manifold $S$. 
Then $H$ acts on $G \times S$ via $(h,(g,s)) \mapsto (gh^{-1}, hs)$ and we denote $G \times^H S = (G \times S) /H$.
Since the $H$-action on $G$ is proper and free, the space $G \times^H S$ is a real manifold.

\begin{definition}
Let $X$ be a CR manifold with a CR action of a Lie group $G$ such that $\Complex \underline{\mathfrak{g}} \cap (T^{1,0}X\oplus T^{0,1} X) = \{ 0 \}$. 
We say that the $G$-action admits CR slices if for every $x \in X$, there exists a real submanifold $S$ of $X$ with $x \in S$ such that
\begin{itemize}
\item The set $S$ is $G_x$-invariant
\item We have $T^{1,0} X \oplus T^{0,1}X  \subset \Complex TS$
\item The map $G \times^{G_x} S \rightarrow X$, $[g,s] \mapsto gs$ is a diffeomorphism onto an open subset of $X$
\end{itemize}
\end{definition}

Note that in the definition above, $S$ has a CR structure induced by $T^{1,0}X$, which defines a CR structure on $G \times^{G_x} S$ such that the map $G \times^{G_x} S \rightarrow X$ is a CR isomorphism onto an open subset.

\begin{proposition}
Let $X$ be a CR manifold of dimension $2n+d$ with CR dimension $d$. 
Let $G$ be a compact Lie group of dimension $k$ with locally free CR action on $X$ which satisfies $\Complex \underline{\mathfrak{g}} \cap (T^{1,0}X\oplus T^{0,1} X) = \{ 0 \}$ and admits CR slices. Then $X/G$ is a CR orbifold of dimension $2n+d-k$ with CR codimension $d-k$.
\end{proposition}
\begin{proof}
This is now analogous to the real version. The charts are given by $G \times^{G_x} S /G = S/G_x$ with $S$ being a CR manifold as above and $G_x$ finite.
\end{proof}

\begin{theorem}\label{t-gue181015}
Let $X$ be an effective CR orbifold. Then there exists a CR manifold $Y$ and a compact Lie group $G$ with locally free CR action on $Y$ which satisfies $\Complex \underline{\mathfrak{g}} \cap (T^{1,0}Y\oplus T^{0,1} Y) = \{ 0 \}$ and admits CR slices such that $Y/G = X$.
\end{theorem}
\begin{proof}
We will first discuss the general construction of the frame bundle.
Let $Z$ be a CR manifold of dimension $2n+d$ with an effective CR action of a finite group $\Gamma$.\\
Let $g$ be a $\Gamma$-invariant metric on $Z$ and define the frame bundle over $Z$ by $Fr(Z) = \{ (x,B) \, | \, x \in Z, \, B$ orthonormal basis in $T_x Z \}$.\\
Then $\Gamma$ acts on $Fr(Z)$ via $(g,(x,B)) \mapsto (gx, dg(B))$ and $O(2n+d)$ acts on $Fr(Z)$ via $(A,(x,B)) \mapsto (x,BA^{-1})$.
Since the $\Gamma$-action on $Z$ is effective, we have that $Fr(Z) / \Gamma$ is a manifold and the $O(2n+d)$-action extends onto this quotient. Note that $(Fr(Z)/ \Gamma )/ O(2n+d) = Z/ \Gamma$.

We may equip $Fr(Z)$ with a CR structure as follows. Take a local trivialization $U \times O(2n+d)$ for the frame bundle, which is equipped with the CR structure coming from $U \subset Z$ and the trivial structure on $O(2n+d)$. One checks that this gives a global CR structure on the frame bundle such that the $\Gamma$ and $O(2n+d)$-actions are CR, therefore $Fr(Z) / \Gamma$ is also a CR manifold.\\
From the construction, we conclude that 
\begin{gather*}
\Complex \cdot \underline{\mathfrak{o}(2n+d)} \cap (T^{1,0} Fr(Z)/ \Gamma \oplus T^{0,1} Fr(Z)/\Gamma) = \{ 0 \}.
\end{gather*}
We will check that $Fr(Z) / \Gamma$ admits CR slices for the $O(2n+d)$-action.

For this, let $U \times O(2n+d)$ be a local trivialization of $Fr(Z)$ such that $U$ is $\Gamma$-invariant. 
We have the map $\pi \colon (U \times O(2n+d)) / \Gamma \rightarrow O(2n+d)/ \Gamma$, which is a submersion. 
One may easily check from the construction that $\pi^{-1}(Id)=:S$ is a CR Slice.

Now let $X$ be an effective CR orbifold. Choose a smooth metric $g$ on $X$, which gives rise to a $\Gamma_i$-invariant metric in every chart $V_i$. 
One may now construct the frame bundle in every chart and glue the $Fr(V_i) / \Gamma_i$ together using the orbifold transition functions.
\end{proof}

\subsection{Embedding theorems for CR orbifolds}\label{s-gue180920}

Let $X$ be a CR manifold and fix $p\in X$. Let $x=(x_1,\ldots,x_{2n+1+d})$ be local coordinates of $X$ defined in a neighborhood $U$ of $p$ such that $x(p)=0$ and \eqref{e-gue180722I}, 
\eqref{e-gue180722II}, \eqref{e-gue180722III}, \eqref{e-gue180723}, \eqref{e-gue180724} hold. Put $N_p:=\set{g\in G;\, g\circ p=p}=\set{g_1:=e_0, g_2,\ldots,g_r}$. Let $\hat D\Subset D\Subset U$ be open sets of $p$ as in Lemma~\ref{l-gue180730b} and Lemma~\ref{l-gue180730bI} respectively. We will use the same notations as in Section~\ref{s-gue180722}. Let $C^\infty(X)^G$ denote the set of $G$-invariant smooth functions on $X$. For every $j=1,\ldots,n+1$, and every $k\in\mathbb N$, let 
$\Td f^{(j)}_k(x)\in C^\infty_0(\hat D)$ be as in \eqref{e-gue180801sI}, let
\begin{equation}\label{e-gue180920m}
\Td g^{(j)}_k(x):=\int_G\Td f^{(j)}_k(g\circ x)d\mu(g)\in C^\infty(X)^G
\end{equation}
and set 
\begin{equation}\label{e-gue180920mI}
g^{(j)}_k:=S\Td g^{(j)}_k\in L^2(X)\bigcap{\rm Ker\,}\ddbar_b.
\end{equation}
We can repeat the proofs of Theorem~\ref{t-gue180731r} and Theorem~\ref{t-gue180731w} and deduce 

\begin{theorem}\label{t-gue180920m}
For every $k\in\mathbb N$, we have $g^{(j)}_k\in C^\infty(X)^G\bigcap{\rm Ker\,}\ddbar_b$, $j=1,2,\ldots,n+1$, and 
\begin{equation}\label{e-gue180921}
\begin{split}
&\lim_{k\To+\infty}dg^{(j)}_k(p)=r\pi^{\frac{d}{2}}dz_j,\ \ j=1,2,\ldots,n,\\
&\lim_{k\To+\infty}dg^{(n+1)}_k(p)=r\pi^{\frac{d}{2}}dx_{2n+1}.
\end{split}
\end{equation}
\end{theorem}

From \eqref{e-gue180921}, we conclude that there is a $k_0\in\mathbb N$, such that for every $k\geq k_0$, 
\begin{equation}\label{e-gue180921m}
\mbox{$\set{(dg^{(1)}_k)(p),(dg^{(2)}_k)(p),\ldots,(dg^{(n+1)}_k(p)}$ are linear independent}.
\end{equation}
Put 
\begin{equation}\label{e-gue180921mI}
g^{(j)}_p:=g^{(j)}_{k_0}\in C^\infty_G(X)\bigcap{\rm Ker\,}\ddbar_b,\  \ j=1,2,\ldots,n+1.
\end{equation}
Consider the $G$-invariant CR map
\begin{equation}\label{e-gue180921mII}
\begin{split}
\hat G_p: X&\To\Complex^{n+1},\\
x&\To (g^{(1)}_p(x),g^{(2)}_p(x),\ldots,g^{(n+1)}_p(x))\in\Complex^{n+1}.
\end{split}
\end{equation}
For $x\in X$, put $H_xX:={\rm Re\,}T^{1,0}_xX$ and set 
\begin{equation}\label{e-gue180922}
\hat H_x(X):={\rm span\,}\set{H_xX, T(x)}. 
\end{equation}
From \eqref{e-gue180921m}, we see that the differential 
\[d\hat G_p: T_pX\To T_p\Complex^{n+1}\]
is injective on $\hat H_pX$, that is, $(d\hat G_p)(V)\neq0$, for every $V\in\hat H_pX$. 
From this observation and the inverse function theorem, we get 

\begin{theorem}\label{t-gue180921mp}
With the notations above, there is an open set $V_p\subset U$ of $p$ in $X$ such that the differential of the map $\hat G_p$ is injective  on $\hat H_xX$ at every point $x$ of $V_p$ and $\hat G_p$ is injective on $V_p/G$ in the sense that for every $x, y\in V_p$ with $x\notin\set{g\circ y;\, g\in G}$, we have $\hat G_p(x)\neq\hat G_p(y)$. 
\end{theorem}

Similarly, we can repeat the proof of Corollary~\ref{c-gue180802} with minor changes and deduce that 

\begin{theorem}\label{t-gue180921my}
With the notations used above, there is a $g_p(x)\in C^\infty(X)^G\bigcap{\rm Ker\,}\ddbar_b$ and an open set $\Td V_p\Subset V_p$ of  $p$ such that
\begin{equation}\label{e-gue180921my}
\abs{g_p(x)}\geq\frac{3}{4}r\pi^{\frac{d}{2}},\  \ \mbox{for every $x\in\Td V_p$},
\end{equation}
and  
\begin{equation}\label{e-gue180921myI}
\abs{g_{p}(x)}\leq \frac{1}{2}r\pi^{\frac{d}{2}},\ \ \mbox{for every $x\notin G V_p$}, 
\end{equation}
where $GV_p:=\set{g\circ x;\, g\in G, x\in V_p}$. 
\end{theorem}

For every $x\in X$, let $V_{x}$ and $\Td V_x$ be open sets as in Theorem~\ref{t-gue180921mp} and Theorem~\ref{t-gue180921my} respectively and 
let $g^{(j)}_x\in C^\infty(X)^G\bigcap{\rm Ker\,}\ddbar_b$, $j=1,2,\ldots,n+1$, $g_x\in C^\infty(X)^G\bigcap{\rm Ker\,}\ddbar_b$ be as in \eqref{e-gue180921mII} and Theorem~\ref{t-gue180921my} respectively. Suppose that 
\[X=\Td V_{p_1}\bigcup\Td V_{p_2}\bigcup\cdots\bigcup\Td V_{p_M},\ \ M\in\mathbb N.\]
For every $p_j$, $j=1,2,\ldots,M$, put 
\begin{equation}\label{e-gue180921myII}
\begin{split}
G_{p_j}: X&\To\Complex^{n+2},\\
x&\To (g^{(1)}_{p_j}(x),\ldots,g^{(n+1)}_{p_j}(x),g_{p_j}(x))\in\Complex^{n+2}.
\end{split}
\end{equation}
Consider the $G$-invariant CR map
\begin{equation}\label{e-gue180921myIII}
\begin{split}
G: X&\To\Complex^{M(n+2)},\\
x&\To (G_{p_1}(x), G_{p_2}(x),\ldots,G_{p_M}(x))\in\Complex^{M(n+2)}.
\end{split}
\end{equation}
We can now prove 

\begin{theorem}\label{t-gue180921ma}
With the notations used above, $G$ is injective on $X/G$ in the sense that for every $x, y\in X$ with $x\notin\set{g\circ y;\, g\in G}$, we have $G(x)\neq G(y)$. 
\end{theorem}

\begin{proof}
Fix $x_0\in X$ and $y_0\in X$ with $x_0\notin\set{g\circ y_0;\, g\in G}$. We may assume that $x_0\in\Td V_{p_1}$. 

\begin{itemize}
\item[case I.] If $y_0\notin GV_{p_1}$ : From \eqref{e-gue180921my} and \eqref{e-gue180921myI}, we see that 
\[\abs{g_{p_1}(x_0)}\geq\frac{3}{4}r\pi^{\frac{d}{2}}>\frac{1}{2}r\pi^{\frac{d}{2}}\geq\abs{g_{p_1}(y_0)}.\]
Hence, $G(x_0)\neq G(y_0)$.

\item[case II.] If $y_0\in GV_{p_1}$: Take $g_0\in G$ and $y_1\in V_{p_1}$ so that $g_0\circ y_1=y_0$. In view of Theorem~\ref{t-gue180921mp}, we see that $\hat G_{p_1}(x_0)\neq\hat G_{p_1}(y_1)$. Since $\hat G_{p_1}$ is $G$-invariant, $\hat G_{p_1}(y_0)=\hat G_{p_1}(y_1)$. We deduce that $\hat G_{p_1}(x_0)\neq \hat G_{p_1}(y_0)$ and hence  $G(x_0)\neq G(y_0)$.
\end{itemize}
\end{proof}

From Theorem~\ref{t-gue180921mp} and Theorem~\ref{t-gue180921ma}, we get Theorem~\ref{t-gue180922m}. 

\section{Induced CR structures}\label{Sec:InducedCRStructures}

Let $X$ be a CR manifold of codimension $d$ and $F \colon X \rightarrow \Complex^m$ a CR map which is an embedding. 
For $d = 1$, we have that \(F\) is a CR embedding, that is, $F(X)$ is a CR submanifold of $\C^m$, meaning that the CR structure is induced by the surrounding space, and \(d\hat{F}\left(T^{1,0}X\right)=\C T\hat{F}(X)\cap T^{1,0}\C^N\). 
For general $d$, this is not obvious. We conclude with the following result.

\begin{theorem}\label{t-gue181015I}
Let $(X, T^{1,0}X)$ be a $(2n+1+d)$-dimensional compact and orientable CR manifold of codimension $d+1$, $d\geq1$. Assume that $X$ admits a CR action of a $d$-dimensional compact Lie group $G$.  Let $T$ be a globally
defined vector field on $X$ such that $\Complex TX=T^{1,0}X\oplus T^{0,1}X\oplus\Complex T\oplus\Complex\underline{\mathfrak{g}}$, where $\underline{\mathfrak{g}}$ is the space of vector fields on $X$ induced by the Lie algebra of $G$. If $X$ is strongly pseudoconvex in the direction of $T$ and $n\geq 2$, then we can find a $G$-equivariant CR embedding $F \colon X \rightarrow \Complex^m$ into some $G$-representation $\Complex^m$, that is \(F\) is a smooth embedding,  $F(X)$ is a CR submanifold of $\Complex^m$ and \(d\hat{F}\left(T^{1,0}X\right)=\C T\hat{F}(X)\cap T^{1,0}\C^N\).
\end{theorem}
\begin{proof}
We have already seen that we may find an equivariant CR map  $F \colon X \rightarrow \Complex^{m_1}$ which is an embedding. 
According to Theorem~\ref{t-gue180921mp}, we may also find a CR map $E \colon X \rightarrow \Complex^{m_2}$ which is $G$-invariant and $dE_x$ is injective on $\hat H_xX$ for every $x \in X$.

Now consider the embedding $H \colon X \rightarrow \Complex^m$, $H= (F,E)$. One may see from linear algebra that $H(X)$ is a CR submanifold of $\Complex^m$ iff dim$_{\mathbb{R}}(T_y H(X) + i T_y H(X)) = 2n+2d+2$ for all $y \in H(X)$, where $T_yH(X)+iT_yH(X)$ denotes 
the space of all vectors of the form $W+JV$ for $W$ and $V$ in $T_yH(X)$ and $J$ is the standard complex structure on complex space. 
Every $G$-representation extends to a $G^\Complex$-representation and since the $G$-action on $X$ is locally free, the $G$-action on $H(X)$ is locally free.
We conclude that dim$_{\mathbb{R}}\mathfrak{g}^\Complex y = 2d$ for all $y \in H(X)$.
This shows that $T_y H(X) + i T_y H(X)$ has at least dimension $2n+2d$. 
Assume there exists a $y$ such that the dimension is exactly $2n+2d$. 
Write $H(x) =y$, then we have $dE(T(x)) \in dE(T^{1,0}_xX \oplus T^{0,1}_xX)$ and $T(x) \in T^{1,0}_x X \oplus T^{0,1}_x X$, which is a contradiction.
\end{proof}

\end{document}